\documentclass[journal,twoside,web]{ieeecolor}
\usepackage{cite}
\usepackage{amsmath,amssymb,amsfonts}
\usepackage{algorithm}
\usepackage{algpseudocode}
\usepackage{graphicx}
\usepackage{textcomp}
\usepackage{bm}
\usepackage{mathrsfs}
\usepackage{dblfloatfix}
\usepackage{caption}
\newtheorem{definition}{Definition}
\newtheorem{asm}{Assumption}
\newtheorem{thm}{Theorem}
\newtheorem{remark}{Remark}
\newtheorem{lemma}{Lemma}
\newtheorem{cor}{Corollary}
\newtheorem{proposition}{Proposition}
\newtheorem{problem}{Problem}
\usepackage{generic}
\def\BibTeX{{\rm B\kern-.05em{\sc i\kern-.025em b}\kern-.08em
    T\kern-.1667em\lower.7ex\hbox{E}\kern-.125emX}}
\begin{document}
\title{Stabilizer Design for Policy Iteration in Stochastic Linear Quadratic Control: A Spectrum-Assignment Approach}
\author{Xinyu Cao, Bing-Chang Wang,\IEEEmembership{Senior Member, IEEE}, and Ying Cao
\thanks{Xinyu Cao, Bing-Chang Wang, Ying Cao are with the School of Control Science and Engineering, Shandong University, Jinan 250061, P. R. China. (e-mail: caoxinyu@mail.sdu.edu.cn, bcwang@sdu.edu.cn, yingcao@mail.sdu.edu.cn).}}

\maketitle

\begin{abstract}
Policy iteration (PI) is an important reinforcement learning tool for solving optimal control problems which includes an initialization stage, i.e., the search for an initial  stabilizing controller. However, the initialization stage typically relies on complete model information, thereby imposing substantial constraints on the initialization of model-free PI. For stochastic systems with multiplicative noise dependent on state and control, the stability is not ensured by Hurwitz conditions as in the deterministic case, but rather by a Lyapunov-type inequality that incorporates both drift and diffusion terms. Therefore, the corresponding model-free PI initialization problem is more challenging. To this end, a novel spectrum assignment method is proposed to obtain an initial stabilizer for PI in continuous-time indefinite stochastic linear quadratic control. With the help of the Lyapunov-type operator's spectrum, the original system is gradually approximated from the stable auxiliary system by adjusting a cumulative factor, thereby obtaining a stabilizing control gain. 
Furthermore, by leveraging system data and
 adjusting the cumulative factor, we design a model-free algorithm that does not rely on an initial stabilizing policy and can
 achieve optimal control. Finally, simulation results are provided to validate the effectiveness of the proposed methods.

\end{abstract}


\section{Introduction}
In reinforcement learning (RL), an agent improves its policy by interacting with the environment, aiming to learn an optimal policy that maximizes the expected long-term cumulative reward \cite{lewis2009reinforcement, mendel1994reinforcement, sutton2018reinforcement, wang2020reinforcement}. RL has been applied in various fields, including robotics \cite{moroncelli2025duality}, autonomous driving \cite{feng2023dense}, communication networks \cite{sutton2018reinforcement}, and power systems \cite{chen2021reinforcement}, among others. Unlike classical dynamic programming, which relies on a known system model and Bellman optimality equation, RL approximates the optimal policy by interacting with the environment and leveraging system operational data, thus overcoming the challenges posed by unknown or difficult-to-identify models \cite{sutton2018reinforcement}. Furthermore, \cite{bertsekas2019reinforcement} points out that RL, which approximates the optimal policy via a policy function, is better suited to high-dimensional and large-scale problems than dynamic programming. As a result, RL has been widely applicable to solving optimal control problems for large-scale systems with unknown models, including both deterministic and stochastic systems. 

Among existing RL algorithms, policy iteration (PI) and value iteration (VI) are two major iterative approaches. PI exhibits faster convergence, but it usually requires an initial control policy that stabilizes the closed-loop system \cite{luo2020policy,chen2023adaptive}. By contrast, VI avoids the need for an initial stabilizing policy, but it generally needs more iterations to approach the target solution, and thus converges more slowly \cite{luo2017output}.  The works \cite{luo2020policy} and \cite{luo2017output} examine the characteristics of PI and VI algorithms, comparing their convergence rates. Specifically, starting from an initial stabilizing control policy, the PI algorithm alternates between policy evaluation and policy improvement to approach the optimal control law. On the other hand, the VI algorithm can start from an arbitrary initial value function, directly iterate the value function to update the control policy implicitly, and then obtain the optimal control law once the value function converges. This means that, under an initial stabilizing policy, the PI method may achieve the optimal control solution more efficiently than the VI method due to its typically faster convergence rate. It should be emphasized that obtaining a stabilizing control law is difficult, and usually requires complete system information \cite{rami2000linear-du5}. Therefore, how to find an initial stabilizing control for PI without relying on system information has become a key issue, which is precisely the motivation for this study.

Existing research on the stabilizing controller for PI initialization mainly falls into two categories. The first solves the problem by combining VI and PI. For instance, \cite{luo2017output} proposed a multi-step policy analysis framework that strikes a trade-off between VI and PI through multi-step heuristic dynamic programming. Subsequently, \cite{luo2019balancing} balanced VI and PI by introducing a balance parameter, which both accelerates VI and avoids the need for an initial stabilizing control. \cite{gao2022resilient} proposed a successive approximation method, termed hybrid iteration, which does not rely on an initial admissible controller and exhibits faster convergence than VI. The second directly focuses on the construction of initial stabilizers, with representative approaches including homotopy initialization \cite{chen2022homotopic}, auxiliary stable system design \cite{pang2025scaling, li2025cooperative}, and stabilizing control learning based on off-policy/Q-learning \cite{li2026adaptive}. The goal is to provide an initial stabilizer for subsequent PI without fully relying on an exact model. 
The homotopy method first constructs a stable artificial system and then gradually approaches the original system, thereby avoiding reliance on an initial stabilizing policy.  The homotopy idea has been further developed in \cite{feng2020connectivity,chen2023adaptive,li2026adaptive,chen2022homotopic,li2025cooperative,lamperski2020computing}. Specifically,  \cite{chen2023adaptive} obtained a stabilizing control using a homotopic method by varying the iteration step length. \cite{pang2025scaling} used scaling policy iteration to obtain a stabilizing gain. \cite{li2026adaptive} analyzed the stabilizing control problem for discrete-time systems via the discount factor. \cite{fan2025homotopy} constructed a homotopy path and proposed a model-free homotopy PI method to compute a stabilizing control policy. In \cite{li2026adaptive, chen2022homotopic, pang2025scaling, fan2025homotopy}, the problem of finding an initial stabilizer for deterministic systems has been well studied. Furthermore, for stochastic discrete-time systems with additive noise, \cite{cui2024robust} obtained a stabilizing controller by solving linear matrix inequalities through system identification. However, this method relies on model identification, and requires samples of high quality and sufficient richness.

In contrast to the existing work, the problem of finding an initial stabilizer for PI in stochastic systems with multiplicative noise has yet to be thoroughly investigated. Compared with the deterministic case, the stochastic system dynamics include a diffusion term that depends on both the state and the control, rendering the associated stochastic Riccati equation more complex. Unlike deterministic stability, mean-square stability of stochastic systems is governed by both the drift and diffusion terms, not by the drift term alone. Therefore, the method in \cite{chen2022homotopic}, which constructs an artificial stable system based on Hurwitz stability in deterministic systems, is no longer valid. Furthermore, existing stochastic optimal control methods typically address standard settings with positive (semi‑)definite cost weighting matrices. However, when the state and control weighting matrices are allowed to be indefinite, the standard linear quadratic (LQ) RL theory offers no direct guarantee for either the stabilizing solution of the generalized Riccati equation or the preservation of closed‑loop stability during iteration. 
Inspired by the above discussion, this paper investigates the  indefinite stochastic linear quadratic (SLQ) control problem with state- and control-dependent multiplicative noise from both model-based and model-free perspectives. For the model-based aspect, we first consider the solvability of the stochastic algebraic Riccati equation (SARE) and the stability during PI for the SLQ problem in the indefinite weighting case. Furthermore, to ensure mean-square stability, we introduce a Lyapunov-type operator and leverage the relationship between its spectrum and mean-square stability to construct an auxiliary stable system. The poles of the operator spectrum are then gradually shifted to the left half-plane by adjusting a cumulative factor, thereby stabilizing the original closed-loop system. For the model-free aspect, we construct data-driven policy evaluation equations and update the policy to obtain a stabilizing policy and achieve the optimal solution to the indefinite SLQ control problem. The main contributions are given as follows:

1) For the indefinite SLQ problem, we first prove that the indefinite PI method preserves mean-square stabilizability and exhibits monotonic convergence in Theorem~1. Then, a model-based RL algorithm containing two phases is proposed.  Phase~I gives an initial stabilizing controller via  spectrum assignment and arbitrary auxiliary positive definite cost weights. Phase~II is to execute indefinite PI to obtain the optimal control using stabilizer from the first phase. 

2) A model-free two-phase PI algorithm is proposed. The method introduces a cumulative factor to construct the state integral data matrix and transforms the PI into a linear regression problem, thereby obtaining the optimal control without requiring the system matrices.

The remainder of this paper is organized as follows. 
Section~\ref{II} introduces the problem formulation and preliminaries. 
Section~\ref{IV} presents the model-based  PI algorithm via spectral-assignment. 
Section~\ref{V} develops the  data-driven model-free algorithm. 
Section~\ref{VI} gives simulation results. 
Section~\ref{VII} concludes the paper.

\textbf{Notation.}
Let \(\mathbb R\) and \(\mathscr C\) denote the sets of real and complex numbers, respectively. 
For a matrix \(X\), \(X^\top\), \(\operatorname{Tr}(X)\), \(\|X\|_F\), and \(\|X\|\) denote its transpose, trace, Frobenius norm, and induced 2-norm, respectively. 
For a symmetric matrix \(X\), \(X>0\) \((X\ge 0)\) means that \(X\) is (semi-)positive definite. For \(n\in\mathbb N\), \(I_n\) denotes the \(n\times n\) identity matrix. The identity operator on \(\mathbb S^n\) is denoted by $\mathcal I_{\mathbb S^n}:\mathbb S^n\to\mathbb S^n,
\
\mathcal I_{\mathbb S^n}(Z)=Z.$ 
These two objects are distinguished throughout the paper.
The set of  symmetric \(n\times n\) matrices is denoted by \(\mathbb S^n\). 
For \(X=X^\top\), \(\lambda_{\min}(X)\) and \(\lambda_{\max}(X)\) denote its minimum and maximum eigenvalues. 
For a square matrix or a linear operator \(\mathcal A\), \(\sigma(\mathcal A)\) denotes its spectrum, and $s(\mathcal A):=\max_{\lambda\in\sigma(\mathcal A)}\operatorname{Re}(\lambda)$ 
denotes its spectral abscissa. 
The open left-half complex plane is denoted by $\mathscr C^-:=\{z\in\mathscr C:\operatorname{Re}(z)<0\}.$ 
The Kronecker product is denoted by \(\otimes\). 
For a matrix \(X\), \(\operatorname{vec}(X)\) stacks its columns into a vector. 
For a symmetric matrix \(S\in\mathbb R^{p\times p}\),
\(\operatorname{vech}(S)\) stacks its lower triangular entries. 
For \(z\in\mathbb R^p\), let $\operatorname{vecv}(z):=D_p^\top\operatorname{vec}(zz^\top),\ \operatorname{vecv}(z)^\top \operatorname{vech}(S) = z^\top Sz, S = S^\top, (x \otimes v)^\top \operatorname{vec}(M) = v^\top M x, M \in \mathbb{R}^{m \times n},$
where \(D_p\) is the duplication matrix satisfying
\(\operatorname{vec}(S)=D_p\operatorname{vech}(S)\) for every
\(S=S^\top\in\mathbb R^{p\times p}\).  
The expectation operator is denoted by \(\mathbb E[\cdot]\).

\section{Problem formulation and preliminaries}
\label{II}

\subsection{Problem formulation}
\label{subsec:problem-formulation}

Consider a continuous-time stochastic system
\begin{equation}
\mathrm{d}x(t)
=
\bigl[Ax(t)+Bu(t)\bigr]\mathrm{d}t
+
\bigl[Cx(t)+Du(t)\bigr]\mathrm{d}w(t),
\label{1-2}
\end{equation}
where \(x(t)\in\mathbb R^n\), \(u(t)\in\mathbb R^m\), and
\(x(0)=x_0\in\mathbb R^n\). The matrices $(A,B;C,D)\in
\mathbb R^{n\times n}\times\mathbb R^{n\times m}
\times\mathbb R^{n\times n}\times\mathbb R^{n\times m}$ 
are constant, and \(w(\cdot)\) is a scalar standard Wiener process defined
on a complete filtered probability space
\((\Omega,\mathcal F,\mathbb P,\{\mathcal F_t\}_{t\ge0})\).

For \(K\in\mathbb R^{m\times n}\), define
\[
A_K:=A-BK,
\qquad
C_K:=C-DK,
\]
and introduce the stochastic Lyapunov operator and its adjoint on
\(\mathbb S^n\):
\begin{align}
\mathcal L_K(Z)
&:=
A_KZ+ZA_K^\top+C_KZC_K^\top,
&&Z\in\mathbb S^n,
\label{eq:LK-def}\\
\mathcal L_K^*(P)
&:=
A_K^\top P+PA_K+C_K^\top PC_K,
&&P\in\mathbb S^n.
\label{eq:LK-adjoint-def}
\end{align}

\begin{definition}[Mean-square stabilizing controller]
\label{def:ms-stabilizingcontroller-feedback}
A gain \(K\in\mathbb R^{m\times n}\) is mean-square stabilizing if
\begin{equation}
s(\mathcal L_K)<0.
\label{stability-condition}
\end{equation}
The set of all mean-square stabilizing gains is denoted by
\begin{equation}
\mathscr K_{\rm ms}
:=
\left\{
K\in\mathbb R^{m\times n}:s(\mathcal L_K)<0
\right\}.
\label{eq:Kms}
\end{equation}
\end{definition}
\begin{asm}
\label{ass:ms-stabilizable}
The system \eqref{1-2} is mean-square stabilizable, namely,
\(\mathscr K_{\rm ms}\neq\varnothing\).
\end{asm}
\begin{definition}[Admissible control]
\label{def:admissible-control}
For any initial state \(x_0\), let
\(\mathcal U_{\rm ad}(x_0)\) be the set of progressively measurable controls
\(u:[0,\infty)\to\mathbb R^m\) such that
\[
\mathbb E\int_0^T\|u(t)\|^2\,\mathrm{d}t<\infty,
\ \forall T>0,
\]
\eqref{1-2} admits a unique solution, and
\begin{equation}
\lim_{t\to\infty}\mathbb E\|x(t)\|^2=0,
\
\mathbb E\int_0^\infty
\left(
\|x(t)\|^2+\|u(t)\|^2
\right)\mathrm{d}t<\infty.
\label{eq:admissible-control}
\end{equation}
\end{definition}

For every \(K\in\mathscr K_{\rm ms}\), the feedback
\(u_K(t):=-Kx(t)\) is admissible. Indeed, \eqref{stability-condition}
implies mean-square exponential stability of the closed-loop system
\begin{equation}\label{sys-k}
    \mathrm{d}x(t)=A_Kx(t)\,\mathrm{d}t+C_Kx(t)\,\mathrm{d}w(t),
\end{equation}
which yields
\[
\mathbb E\int_0^\infty\|x(t)\|^2\,\mathrm{d}t<\infty,
\quad
\mathbb E\int_0^\infty\|u_K(t)\|^2\,\mathrm{d}t<\infty.
\]

The objective of this paper is to solve the following model-free
indefinite SLQ control problem, which can be formulated as follows.

Define the associated performance index
\begin{equation}
J(x_0,u)
:=
\mathbb E\int_0^\infty
\left[
x(t)^\top Qx(t)+u(t)^\top Ru(t)
\right]\mathrm{d}t,
\label{10-9}
\end{equation}
where $Q=Q^\top\in\mathbb S^n,
\
R=R^\top\in\mathbb S^m$ 
are allowed to be indefinite. For every \(u\in\mathcal U_{\rm ad}(x_0)\), 
and hence \(J(x_0,u)\) is finite on \(\mathcal U_{\rm ad}(x_0)\).

\begin{problem}[Model-free indefinite SLQ control]
\label{prob:indefinite-slq}
Without access to the system matrices \(A,B;C,D\), find an admissible optimal controller
\[
u^*(t)=-K^*x(t), K^*\in \mathscr K_{\rm ms}
\]
such that
\begin{equation}
\inf_{u\in\mathcal U_{\rm ad}(x_0)}J(x_0,u)
=
J(x_0,u^*),
\label{eq:SLQ-problem}
\end{equation}
\end{problem}

The solution to  Problem~\ref{prob:indefinite-slq} is associated with the following SARE
\begin{equation}
A^\top P+PA+C^\top PC+Q
-L(P)^\top G(P)^{-1}L(P)=0,
\label{eq:sare}
\end{equation}
where
\begin{align}
G(P)&:=R+D^\top PD,
\label{eq:G-P}\\
L(P)&:=B^\top P+D^\top PC.
\label{eq:L-P}
\end{align}
    \begin{definition}[Stabilizing SARE solution]
\label{def:stabilizing-sare}
A matrix \(P=P^\top\) is a stabilizing solution to \eqref{eq:sare} if
\[
G(P)>0,
\quad
K(P):=G(P)^{-1}L(P)\in\mathscr K_{\rm ms}.
\]
Define
\begin{equation}
\mathcal M:=
\left\{
P=P^\top:
G(P)>0,\quad H(P)\ge0
\right\},
\label{eq:M-def}
\end{equation}
where
\begin{equation}
H(P):=
\begin{bmatrix}
A^\top P+PA+C^\top PC+Q & L(P)^\top\\
L(P) & G(P)
\end{bmatrix}.
\label{eq:H-def}
\end{equation}
\end{definition}

Since \(Q\) and \(R\) may be indefinite, \(G(P)\) is not necessarily
positive definite. Therefore, \eqref{eq:sare} is considered together with
the constraint \(G(P)>0\). To this end, we give the following assumption.

\begin{asm}
\label{ass:riccati-feasibility}
The set \(\mathcal M\) has a nonempty interior. Equivalently, there exists
a matrix \(\widetilde P=\widetilde P^\top\) such that $G(\widetilde P)>0,
\
H(\widetilde P)>0.$
\end{asm}

\begin{lemma}[{\cite[Theorems~4.2--4.3]{rami2001solvability}}]
\label{lem:stabilizing-sare}
Under Assumptions~\ref{ass:ms-stabilizable} and
\ref{ass:riccati-feasibility}, the SARE \eqref{eq:sare} has a unique
stabilizing solution \(P^*=P^{*\top}\), the optimal control gain is
\begin{equation}
K^*:=G(P^*)^{-1}L(P^*) \in \mathscr K_{\rm ms}.
\label{eq:Kstar}
\end{equation}
\end{lemma}

To obtain the optimal solutions in  \eqref{eq:sare} iteratively, the model-based indefinite PI is proposed in the following theorem:
\begin{thm}[Indefinite PI]
\label{thm:indefinite-PI}
Suppose that Assumptions~\ref{ass:ms-stabilizable} and
\ref{ass:riccati-feasibility} hold. Let \(K^0\in\mathscr K_{\rm ms}\).
For \(k=0,1,2,\ldots\), compute
\begin{align}
\mathcal L_{K^k}^*(P^k)+Q_{k}&=0,
\quad
Q_{k}:=Q+(K^k)^\top RK^k,
\label{eq:indefinite-PE}\\
K^{k+1}&=G(P^k)^{-1}L(P^k).
\label{eq:indefinite-PI}
\end{align}
Then, the following hold:
\begin{enumerate}
\item $P^k-P^*\ge0,
\quad
G(P^k)\ge G(P^*)>0,
\quad \forall k\ge0;$ 

consequently, \(K^{k+1}\) in \eqref{eq:indefinite-PI} is well defined;

\item \(K^k\in\mathscr K_{\rm ms}\) for every \(k\ge0\);

\item $P^*\leq P^{k+1}\leq P^k,
\quad \forall k\ge0;$
\item $P^k\to P^*,
\quad
K^k\to K^*,
\quad k\to\infty.$

\end{enumerate}
\end{thm}

\begin{proof}
See Appendix~\ref{thm1+prop2}.
\end{proof}
\begin{remark}
    A Lyapunov-based construction of a stabilizing gain $K^0$ requires 
stability analysis. However, the original indefinite weights \(Q\) and \(R\) may not ensure \(Q+K^\top RK >0\) holds.
Therefore, we introduce auxiliary positive definite weights $Q_{\rm a},R_{\rm a}$ and employ \(Q_{\rm a}+K^\top R_{\rm a}K>0\) solely for the stabilizer construction. The justification for this approach is that, by Definition~\ref{def:ms-stabilizingcontroller-feedback}, \(K\in\mathscr K_{\rm ms}\) depends only on the closed-loop dynamics $(A-BK,C-DK)$ and is independent of the weights \(Q,R\). 

The PI algorithm for Problem~\ref{prob:indefinite-slq}, however, requires an
initial stabilizing gain \(K^0\in\mathscr K_{\rm ms}\), which is generally
unavailable when the system matrices are unknown. A two-phase algorithm
is therefore proposed: Phase~I uses the auxiliary positive definite
weights to construct such a stabilizer \(K^\dagger\in\mathscr K_{\rm ms}\), then  Phase~II sets
\(K^0:=K^\dagger\) and applies PI to solve the original indefinite SLQ
Problem~\ref{prob:indefinite-slq}.
\end{remark}

In Theorem~\ref{thm:indefinite-PI}, the initial stabilizing gain $K^0$ is applied to initialize the PI procedure. This requirement is routine in the model-based setting but becomes nontrivial when system information is completely unknown. To circumvent this obstacle,  
this paper proposes the stabilizer design via a spectrum-assignment approach, which converts any given control gain into a stabilizing one from a new perspective. Then, we  design novel PI algorithms that operate in two phases, applicable to both model-based and model-free settings, to address Problem~\ref{prob:indefinite-slq}. Specifically, the main algorithms are carried out sequentially in the following two phases:
\begin{enumerate}
\item \textit{Phase I (stabilizer construction).} 
Starting from an arbitrary \(K_{\rm a}^0\), a
spectrum-assignment-based PI procedure is developed to construct
a gain $K^\dagger\in\mathscr K_{\rm ms}.$ To perform this stage, we introduce
positive definite auxiliary weights in Phase I
\[
Q_{\rm a}=Q_{\rm a}^\top>0,
\qquad
R_{\rm a}=R_{\rm a}^\top>0.
\]
A spectral-translation mechanism is used to construct $K^\dagger$. These
auxiliary quantities are used exclusively for stabilizer construction in Phase I and
do not enter the performance index \eqref{10-9}.

\item \textit{Phase II (indefinite SLQ optimization).} 
The stabilizer \(K^\dagger\) obtained in Phase~I is served as the initial gain $K^0:=K^\dagger.$ 
It then provides the admissible initialization required by
Theorem~\ref{thm:indefinite-PI}. Starting from \(K^0\), the corresponding
PI obtains the approximate  solutions \((P^*,K^*)\) to the original indefinite SLQ
Problem~\ref{prob:indefinite-slq}.
\end{enumerate}
\begin{remark}
    By Definition~\ref{def:ms-stabilizingcontroller-feedback}, a stabilizing controller \(K\) is only associated with the dynamics $A,B,C,D$ and does not involve weights \(Q,R\). Therefore, we introduce  positive definite costs $Q_a,R_a$, thereby providing the shifted auxiliary closed-loop system with a stability margin. Notably, these weights serve exclusively to construct the stabilizer in Phase I and are not utilized in Phase II.
\end{remark}

\subsection{Preliminaries}
\label{subsec:preliminaries}

For a symmetric matrix \(P\) satisfying \(G(P)>0\), define
\[
K(P):=G(P)^{-1}L(P).
\]
For \(K\in\mathbb R^{m\times n}\), let
\begin{align}
Q_K&:=Q+K^\top RK,
\label{P1.3}\\
\Phi(P,K)&:=\mathcal L_K^*(P)+Q_K.
\label{P1.2}
\end{align}

\begin{lemma}[Policy identity]
\label{lem:policy-identity}
Let \(P=P^\top\) satisfy \(G(P)>0\). Then, for every
\(K\in\mathbb R^{m\times n}\),
\begin{equation}
\Phi(P,K)-\Phi(P,K(P))
=
\bigl(K-K(P)\bigr)^\top
G(P)
\bigl(K-K(P)\bigr).
\label{P1.4}
\end{equation}
\end{lemma}

\begin{proof}
See Appendix~\ref{app:proofs}.
\end{proof}

\begin{lemma}[Positive inverse of a stable Lyapunov operator]
\label{lem:positive-lyapunov-inverse}
Let \(K\in\mathscr K_{\rm ms}\). For every \(Y=Y^\top\ge0\), the
equation
\begin{equation}
\mathcal L_K^*(X)+Y=0
\label{eq:positive-lyapunov}
\end{equation}
admits a unique solution \(X=X^\top\ge0\).
\end{lemma}

\begin{proof}
See Appendix~\ref{lem:ct_order1}.
\end{proof}

\begin{lemma}[{\cite{zhang2004stabilizability}}]
\label{lem:instability-certificate}
If \(K\notin\mathscr K_{\rm ms}\), then there exist a scalar
\(\mu\ge0\) and a nonzero matrix \(X=X^\top\ge0\) such that
\begin{equation}
\mathcal L_K(X)=\mu X.
\label{eq:spectral_obstruction}
\end{equation}
\end{lemma}


For \(\bar r>0\), \(c\in\mathbb R\), and
\(K\in\mathbb R^{m\times n}\), define
\begin{align}
\mathcal L^{\rm a}_{K,c}(Z)
:={}&
\left(A-\frac{\bar r}{2}I_n+cI_n-BK\right)Z \nonumber\\
&+Z\left(A-\frac{\bar r}{2}I_n+cI_n-BK\right)^\top \nonumber\\
&+(C-DK)Z(C-DK)^\top.
\label{eq:aux-operator}
\end{align}

\begin{proposition}[Exact translation of the stochastic Lyapunov spectrum]
\label{prop:operator-translation}
For every \(\bar r>0\), \(c\in\mathbb R\), and
\(K\in\mathbb R^{m\times n}\),
\begin{equation}
\mathcal L^{\rm a}_{K,c}
=
\mathcal L_K+(-\bar r+2c)\mathcal I_{\mathbb S^n}.
\label{eq:operator-translation}
\end{equation}
Consequently,
\begin{align}
\sigma\bigl(\mathcal L^{\rm a}_{K,c}\bigr)
&=
\sigma(\mathcal L_K)-\bar r+2c,
\label{eq:spectrum-translation}\\
s(\mathcal L_K)
&=
s\bigl(\mathcal L^{\rm a}_{K,c}\bigr)+\bar r-2c.
\label{eq:abscissa-translation}
\end{align}
\end{proposition}

\begin{proof}
For every \(Z\in\mathbb S^n\),
\[
\mathcal L^{\rm a}_{K,c}(Z)-\mathcal L_K(Z)
=
(-\bar r+2c)Z.
\]
This proves \eqref{eq:operator-translation}. Since the identity operator
\(\mathcal I_{\mathbb S^n}\) commutes with every linear operator on
\(\mathbb S^n\), \eqref{eq:spectrum-translation} and
\eqref{eq:abscissa-translation} follow immediately.
\end{proof}

\section{Model-based PI design via spectral criteria}\label{IV}
In this section, we propose a model-based two-phase algorithm to solve Problem~\ref{prob:indefinite-slq}. In Phase~I, auxiliary weights $Q_a>0,R_a>0$ are adopted to find a stabilizing control policy satisfying $\sigma(\mathcal L_{K_{\rm a}^k}) \subset \mathscr{C}^{-}$.  
Then, in Phase~II, we revert to the original indefinite weights $Q, R$ in \eqref{eq:sare}, and initialize the PI with the stabilizing gain $K^\dagger:=K_a^k$ obtained in Phase~I to iteratively approximate the optimal solution. The theoretical results in this section will guide the model-free design in the next section.
We impose the following assumption.

\begin{asm}[Model-based spectral bound]
\label{ass:model-based-shift}
For the model-based Phase~I design, let 
\begin{equation}
r_{\max}:=\max\left\{0,\,
s\!\left(\mathcal L_{K_{\rm a}^0}\right)
\right\}, \label{eq:rmax}
\end{equation}
and choose
\begin{equation}
\bar r:=r_{\max}+\beta,\quad \beta>0. \label{eq:rbar}
\end{equation}
\end{asm}
\begin{lemma}[Initial auxiliary stability]
\label{lem:initial-auxiliary-stability}
Let \(K_{\rm a}^0\in\mathbb R^{m\times n}\) be arbitrary and 
choose \(\beta>0\),  \(\alpha_0>0\) such that \(2\alpha_0<\beta\). Let
\(c_0:=\alpha_0\). Then
\[
s\!\left(
\mathcal L^{\rm a}_{K_{\rm a}^0,c_0}
\right)<0.
\]
\end{lemma}

\begin{proof}
By Proposition~\ref{prop:operator-translation},
\[
s\!\left(
\mathcal L^{\rm a}_{K_{\rm a}^0,c_0}
\right)
=
s\!\left(\mathcal L_{K_{\rm a}^0}\right)
-\bar r+2\alpha_0\le r_{\rm max}-\bar r+2\alpha_0<0.
\]
\end{proof}

For the auxiliary design in Phase~I, the cumulative factor is denoted by
\[
c_k:=\sum_{j=0}^{k}\alpha_j,\quad c_0=\alpha_0>0.
\]
At iteration \(k\), define the auxiliary closed-loop operator as $\mathcal L^{\rm a}_{K_{\rm a}^k,c_k}.$

The auxiliary weights \(Q_{\rm a}, R_{\rm a}>0\) are used throughout Phase I. The original indefinite weights \(Q\) and \(R\) are used only after a stabilizing gain $K^{\dagger}$ has been obtained, i.e. Phase~II.

We now propose a two-phase model-based  algorithm.
\begin{algorithm}[t]
\caption{Two-phase model-based PI for indefinite SLQ control}
\label{alg:model-based}
\begin{algorithmic}[1]
\State \textbf{Input:} \(A,B,C,D\); original weights \(Q=Q^\top,R=R^\top\); auxiliary weights \(Q_{\rm a}>0,R_{\rm a}>0\); \(\beta>0\), \(\alpha_0>0\) with \(2\alpha_0<\beta\); tolerance \(\varepsilon\ge0\).
\State \textbf{Output:} An initial stabilizer \(K^\dagger\), whenever the Phase-I stopping condition is reached, and the stabilizing solution \((P^*,K^*)\) of the original SARE \eqref{eq:sare}.
\State Set \(\bar r=r_{\max}+\beta\), \(K_{\rm a}^0=0\), \(c_0=\alpha_0\), and \(k=0\), compute \(r_{\max}=\max\{s(\mathcal L_{K_{\rm a}^0}),0\}\).
\Statex \textbf{Phase I: initial stabilizer $K^\dagger$ construction.}
\While{\(\bar r-2c_k\ge-\varepsilon\)}
    \State Solve \(P_{\rm a}^k\) from \eqref{eq:auxiliary-PE}.
    \State Update \(K_{\rm a}^{k+1}\) from \eqref{eq:auxiliary-PI}.
    \State Choose \(\alpha_{k+1}\) according to \eqref{eq:model-based-alpha}.
    \State Set \(c_{k+1}=c_k+\alpha_{k+1}\), and \(k\gets k+1\).
\EndWhile
\State Set \(K^\dagger=K_{\rm a}^k\).
\Statex \textbf{Phase II: original indefinite SLQ problem.}
\State Set \(K^0=K^\dagger\) and \(i=0\).
\Repeat
    \State Solve \eqref{eq:indefinite-PE} for $P^i$.
    \State Update $K^{i+1}$ with \eqref{eq:indefinite-PI}.
    \State Set \(i\gets i+1\).
\Until{the prescribed convergence criterion is satisfied}
\State \Return \(P^*=P^i\), \(K^*=K^i\).
\end{algorithmic}
\end{algorithm}

\begin{lemma}[Auxiliary spectrum-assignment PI]
\label{lem:auxiliary-spectrum-assignment}
Suppose that Assumption~\ref{ass:model-based-shift} holds. Let
\(Q_{\rm a}=Q_{\rm a}^{\top}>0\) and
\(R_{\rm a}=R_{\rm a}^{\top}>0\). Initialize
\(K_{\rm a}^{0}=0\) and \(c_0=\alpha_0\), where
\(0<2\alpha_0<\beta\), and let \(\varepsilon\ge0\). For
\(k=0,1,2,\ldots\), let \(P_{\rm a}^{k}=P_{\rm a}^{k\top}\), solve
\begin{equation}
\begin{aligned}
0={}&
\left(A-\frac{\bar r}{2}I_n+c_kI_n-BK_{\rm a}^k\right)^\top
P_{\rm a}^k\\
&+P_{\rm a}^k
\left(A-\frac{\bar r}{2}I_n+c_kI_n-BK_{\rm a}^k\right)\\
&+(C-DK_{\rm a}^k)^\top P_{\rm a}^k(C-DK_{\rm a}^k)\\
&+Q_{\rm a}+(K_{\rm a}^k)^\top R_{\rm a}K_{\rm a}^k,
\end{aligned}
\label{eq:auxiliary-PE}
\end{equation}
and update
\begin{equation}
K_{\rm a}^{k+1}
=
\left(R_{\rm a}+D^\top P_{\rm a}^kD\right)^{-1}
\left(B^\top P_{\rm a}^k+D^\top P_{\rm a}^kC\right).
\label{eq:auxiliary-PI}
\end{equation}
Fix \(\vartheta\in(0,1)\), and select
\begin{equation}
\alpha_{k+1}
=
\frac{\vartheta\lambda_{\min}(Q_{\rm a})}
{2\lambda_{\max}(P_{\rm a}^{k})},
\quad
c_{k+1}=c_k+\alpha_{k+1}.
\label{eq:model-based-alpha}
\end{equation}
Then the following statements hold.

\begin{enumerate}
\item For every \(k\ge0\),
\begin{align}
s\bigl(\mathcal L^{\rm a}_{K_{\rm a}^k,c_k}\bigr)&<0,
\label{eq:auxiliary-stability-all-k}\\
s\bigl(\mathcal L_{K_{\rm a}^k}\bigr)&<\bar r-2c_k.
\label{eq:original-spectrum-bound}
\end{align}
Moreover, the step length in \eqref{eq:model-based-alpha} satisfies
\begin{equation}
0<\alpha_{k+1}
<
-\frac12s\bigl(\mathcal L^{\rm a}_{K_{\rm a}^{k+1},c_k}\bigr).
\label{eq:model-based-alpha-feasibility}
\end{equation}
Consequently, if the stopping condition
\begin{equation}
\bar r-2c_{k_\star}<-\varepsilon\leq 0
\label{eq:phase-I-stop}
\end{equation}
holds for an index \(k_\star\), then $K^\dagger:=K_{\rm a}^{k_\star}\in\mathscr K_{\rm ms},
\
s(\mathcal L_{K^\dagger})<-\varepsilon.$

\end{enumerate}
\end{lemma}

\begin{proof}
See Appendix~\ref{auxiliary-spectrum-assignment}.
\end{proof}
\begin{proposition}\label{alpha_bound+Nmax}
     Suppose that the Phase~I value matrices are uniformly
bounded, i.e., there exists
\(\bar p>0\) such that
\begin{equation}
\lambda_{\max}(P_{\rm a}^k)\le\bar p.
\label{eq:uniform-PhaseI-value-bound}
\end{equation}
 Then \(\alpha_{k+1}\) has the uniform
positive lower bound
\begin{equation}
\alpha_{k+1}\ge
\underline\alpha:=
\frac{\vartheta\lambda_{\min}(Q_{\rm a})}{2\bar p}>0.
\label{eq:alpha-lower-bound}
\end{equation}
Thus, Phase~I reaches \eqref{eq:phase-I-stop} after finitely many
iterations. More precisely, if
\(c_0\le(\bar r+\varepsilon)/2\), then the stopping condition is reached
after at most
\begin{equation}
N_{\max}:=
\left\lfloor
\frac{(\bar r+\varepsilon)/2-c_0}{\underline\alpha}
\right\rfloor+1
\label{eq:PhaseI-iteration-bound}
\end{equation}
additional iterations.
\end{proposition}
\begin{proof}
See Appendix~\ref{auxiliary-spectrum-assignment}.
\end{proof}
\begin{remark}
\label{rem:fixed-K-translation}
Lemma~\ref{lem:auxiliary-spectrum-assignment} is used to construct a stabilizer: whenever \eqref{eq:phase-I-stop} is reached, the returned gain is a stabilizer of the original system. The spectrum-translation identity \eqref{eq:spectrum-translation} is used only for the same fixed feedback gain. More precisely, the auxiliary policy update first changes \(K_{\rm a}^k\) to \(K_{\rm a}^{k+1}\) while keeping \(c_k\) fixed; the cumulative factor is subsequently changed from \(c_k\) to \(c_{k+1}\) while keeping \(K_{\rm a}^{k+1}\) fixed. Therefore, the variation of the feedback gain across policy iterations does not affect the validity of Proposition \ref{prop:operator-translation}. The diffusion term \((C-DK)Z(C-DK)^\top\) is retained identically in both operators and hence cancels exactly in \eqref{eq:operator-translation}.
\end{remark}

\begin{remark}[Roles of the two phases]
\label{rem:two-phase-implementation}
The stabilizing property of \(K^\dagger\) is entirely determined by the closed-loop dynamics, i.e., the criterion \(s(\mathcal L_{K^\dagger})<0\) depends only on the closed-loop dynamics \((A-BK^\dagger,C-DK^\dagger)\). Phase~I uses the auxiliary positive definite weights \(Q_{\rm a},R_{\rm a}\) only to construct a gain \(K^\dagger\in\mathscr K_{\rm ms}\). Once \(K^\dagger\) is obtained, Phase II discards all auxiliary quantities \(\bar r\), \(c_k\), \(\alpha_k\), \(Q_{\rm a}\), and \(R_{\rm a}\), and uses \(K^\dagger\) to initialize the PI in Theorem~\ref{thm:indefinite-PI} to solve the indefinite SLQ Problem~\ref{prob:indefinite-slq}.
\end{remark}

We provide  Fig.~\ref{fig:operator-spectrum-translation} to illustrate the evolution of the closed-loop poles of both the auxiliary and the original operators over iterations.

\begin{figure*}[!t]
  \centering
\includegraphics[width=0.7\textwidth]{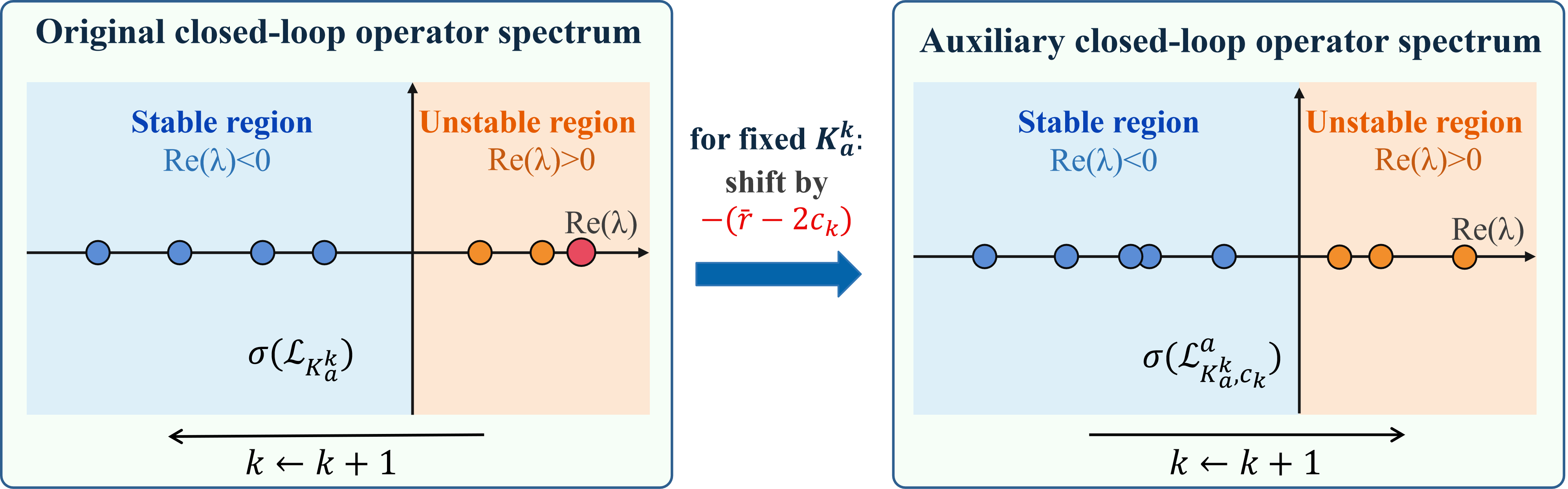}

\caption{Illustration of the operator-spectrum translation used in Phase~I.
  For a fixed feedback gain \(K_{\rm a}^k\), the original closed-loop operator
  \(\mathcal L_{K_{\rm a}^k}\) and the auxiliary closed-loop operator
  \(\mathcal L^{\rm a}_{K_{\rm a}^k,c_k}\) satisfy
  \(\mathcal L^{\rm a}_{K_{\rm a}^k,c_k}
  =\mathcal L_{K_{\rm a}^k}-(\bar r-2c_k)\mathcal I_{\mathbb S^n}\).
  Hence every point in the auxiliary spectrum is obtained by translating the
  corresponding point in the original closed-loop spectrum by
  \(-\left(\bar r-2c_k\right)\) along the real axis.}
  \label{fig:operator-spectrum-translation}
\end{figure*}

Fig.~\ref{fig:operator-spectrum-translation} shows the relation between the
original and auxiliary closed-loop operators at the same feedback gain
\(K_{\rm a}^k\). The left panel depicts the spectrum of
\(\mathcal L_{K_{\rm a}^k}\), whose spectral abscissa
\(s(\mathcal L_{K_{\rm a}^k})\) may lie in the open right half-plane. The right
panel depicts the spectrum of \(\mathcal L^{\rm a}_{K_{\rm a}^k,c_k}\), obtained
from the exact translation identity
\begin{equation}
\mathcal L^{\rm a}_{K_{\rm a}^k,c_k}
=\mathcal L_{K_{\rm a}^k}-(\bar r-2c_k)\mathcal I_{\mathbb S^n}.
\label{eq:fig-operator-translation}
\end{equation}
Consequently,
\begin{equation}
 s\big(\mathcal L_{K_{\rm a}^k}\big)
 =s\big(\mathcal L^{\rm a}_{K_{\rm a}^k,c_k}\big)+\bar r-2c_k.
\label{eq:fig-abscissa-relation}
\end{equation}
Phase~I maintains \(s(\mathcal L^{\rm a}_{K_{\rm a}^k,c_k})<0\) through the
auxiliary positive definite policy update. When the cumulative factor reaches
\(\bar r-2c_k<0\), \eqref{eq:fig-abscissa-relation} gives
\(s(\mathcal L_{K_{\rm a}^k})<0\). The gain \(K_{\rm a}^k\) is then a stabilizer for the original system \eqref{sys-k}.

\section{Model-free PI design via spectral criteria}\label{V}

A stabilizing feedback gain $K^{\dagger}$ is derived from the system matrices in the preceding section. This section develops its data-driven counterpart. Unhatted and hatted variables denote exact expectation-based and finite-sample estimates, respectively.

\subsection{Data-driven search for a sufficient spectral-shift parameter $\bar r$}\label{subsec:shift-search}

\begin{proposition}[Sufficient shift condition]
\label{prop:sufficient-shift}
Fix a candidate \(\bar r>0\), set \(K_{\rm a}^0=0\) and \(c_0=\alpha_0\), and define
\[
A_{{\rm a},0}(\bar r)
:=
A-\frac{\bar r}{2}I_n+\alpha_0I_n.
\]

We solve a symmetric matrix \(P_{\rm a}^0(\bar r)\) satisfying
\begin{equation}
A_{{\rm a},0}(\bar r)^\top P_{\rm a}^0
+
P_{\rm a}^0A_{{\rm a},0}(\bar r)
+
C^\top P_{\rm a}^0C
+
Q_{\rm a}
=0.
\label{eq:initial-auxiliary-lyapunov}
\end{equation}
If \(P_{\rm a}^0(\bar r)>0\), then $s(\mathcal L_0)<\bar r-2\alpha_0.$ In particular, \(\bar r>r_{\max}\).
\end{proposition}

\begin{proof}
Equation~\eqref{eq:initial-auxiliary-lyapunov} gives
\[
A_{{\rm a},0}(\bar r)^\top P_{\rm a}^0
+
P_{\rm a}^0A_{{\rm a},0}(\bar r)
+
C^\top P_{\rm a}^0C
=
-Q_{\rm a}<0.
\]
The stochastic Lyapunov criterion Lemma~2.1 in \cite{li2022stochastic} yields stability of
\((A_{{\rm a},0}(\bar r),C)\). Applying Proposition~\ref{prop:operator-translation} with \(K=0\) and \(c=\alpha_0\) gives
\[
s(\mathcal L_0)-\bar r+2\alpha_0<0.
\]
The stated inequality follows. If \(s(\mathcal L_0)\ge0\), then
\[
\bar r>s(\mathcal L_0)+2\alpha_0\ge r_{\max}.
\]
If \(s(\mathcal L_0)<0\), then \(r_{\max}=0<\bar r\).
\end{proof}
Proposition~\ref{prop:sufficient-shift} provides a characterization of the search for the candidate $\bar r$ in Phase I.   It determines a sufficient spectrum-shift parameter satisfying \(\bar r>r_{\max}\).

In the model-free numerical implementation, for finite Monte-Carlo data, let \(\widehat P_{\rm a}^0(\bar r)\) denote the least-squares estimate. The tests
\[
\lambda_{\min}\!\big(\widehat P_{\rm a}^0(\bar r)\big)\ge\tau_P,
\quad
\sigma_{\min}\!\big(\widehat\Theta_{{\rm a},A}^{0}(\bar r)\big)\ge\tau_\Theta,
\]
and
\[
\left\|
\widehat\Theta_{{\rm a},A}^{0}(\bar r)
\widehat\theta_{{\rm a},0}(\bar r)
-
\widehat b_{{\rm a},0}(\bar r)
\right\|_2
\le\tau_{\rm res}
\]
serve as acceptance criteria for the sequence $\bar r$ in the model-free numerical implementation. 
Here
\(\tau_P,\tau_\Theta,\tau_{\rm res}>0\) are prescribed numerical margins.

\subsection{Auxiliary policy evaluation from data in Phase~I}
\label{subsec:phase-I-regression}

For the \(k\)-th Phase~I iteration, define
\begin{align}
A_{\rm a}^k
&:=A-\frac{\bar r}{2}I_n-BK_{\rm a}^k+c_kI_n,
\label{eq:phase-I-Ak}\\
C_{\rm a}^k
&:=C-DK_{\rm a}^k,
\label{eq:phase-I-Ck}\\
\delta_k
&:=\frac{\bar r}{2}-c_k,
\label{eq:phase-I-delta}\\
v_{\rm a}^k(t)
&:=u(t)+K_{\rm a}^k x(t).
\label{eq:phase-I-virtual-input}
\end{align}
The state equation \eqref{sys-k} admits the decomposition
\begin{equation}
\mathrm dx
=
\big(A_{\rm a}^k x+Bv_{\rm a}^k+\delta_kx\big)\mathrm dt
+
\big(C_{\rm a}^k x+Dv_{\rm a}^k\big)\mathrm dw.
\label{eq:phase-I-split}
\end{equation}

Let \(P_{\rm a}^k\) solve the auxiliary policy-evaluation equation
\begin{equation}
(A_{\rm a}^k)^\top P_{\rm a}^k
+
P_{\rm a}^kA_{\rm a}^k
+
(C_{\rm a}^k)^\top P_{\rm a}^kC_{\rm a}^k
+
Q_{{\rm a},k}
=
0,
\label{eq:phase-I-PE}
\end{equation}
where
\begin{align}
Q_{{\rm a},k}
&:=Q_{\rm a}+(K_{\rm a}^k)^\top R_{\rm a}K_{\rm a}^k,
\label{eq:Qa-k}\\
\Gamma_{\rm a}^k
&:=D^\top P_{\rm a}^kD,
\label{eq:Gamma-a-k}\\
\widetilde K_{\rm a}^{k+1}
&:=B^\top P_{\rm a}^k+D^\top P_{\rm a}^kC
=
\big(R_{\rm a}+\Gamma_{\rm a}^k\big) K_{\rm a}^{k+1},
\label{eq:Ktilde-a-k}\\
\Xi_{\rm a}^{k+1}
&:=\widetilde K_{\rm a}^{k+1}
-\Gamma_{\rm a}^kK_{\rm a}^{k}.
\label{eq:Xi-a-k}
\end{align}
The policy update is
\begin{equation}
K_{\rm a}^{k+1}
=
\big(R_{\rm a}+\Gamma_{\rm a}^k\big)^{-1}
\left(\Xi_{\rm a}^{k+1}+\Gamma_{\rm a}^kK_{\rm a}^k\right).
\label{eq:Ka-update-data}
\end{equation}

Applying It\^{o}'s formula to \(x^\top P_{\rm a}^k x\), using
\eqref{eq:phase-I-split} and \eqref{eq:phase-I-PE}, and integrating over
\([t_\ell,t_{\ell+1}]\) yields
\begin{equation}\label{eq:phase-I-integral-identity}
\begin{split}
    &\mathbb E\!\left[
x(t_{\ell+1})^\top P_{\rm a}^k x(t_{\ell+1})
-
x(t_\ell)^\top P_{\rm a}^k x(t_\ell)
\right]\\
&=
\mathbb E\!\int_{t_\ell}^{t_{\ell+1}}
\bigg[
-x^\top Q_{{\rm a},k}x +2(v_{\rm a}^k)^\top\Xi_{\rm a}^{k+1}x\\
& \quad +(v_{\rm a}^k)^\top\Gamma_{\rm a}^k v_{\rm a}^k
+2\delta_kx^\top P_{\rm a}^k x
\bigg]\mathrm d\tau .
\end{split}
\end{equation}

Define
\begin{align}
\Delta\mathbb X_\ell
&:=
\mathbb E\!\left[
\operatorname{vecv}(x(t_{\ell+1}))
-
\operatorname{vecv}(x(t_\ell))
\right],
\label{eq:delta-X}\\
\mathbb I_{xx}
&:=
\left[
\mathbb E\!\int_{t_0}^{t_1}\operatorname{vecv}(x)\mathrm d\tau,
\ldots,
\mathbb E\!\int_{t_{s-1}}^{t_s}\operatorname{vecv}(x)\mathrm d\tau
\right]^\top,
\label{eq:Ixx}\\
\mathbb I_{{\rm a},xv}^{k}
&:=
\left[
\mathbb E\!\int_{t_0}^{t_1}x\otimes v_{\rm a}^k\mathrm d\tau,
\ldots,
\mathbb E\!\int_{t_{s-1}}^{t_s}x\otimes v_{\rm a}^k\mathrm d\tau
\right]^\top,
\label{eq:Iaxv}\\
\mathbb I_{{\rm a},vv}^{k}
&:=
\left[
\mathbb E\!\int_{t_0}^{t_1}\operatorname{vecv}(v_{\rm a}^k)\mathrm d\tau,
\ldots,
\mathbb E\!\int_{t_{s-1}}^{t_s}\operatorname{vecv}(v_{\rm a}^k)\mathrm d\tau
\right]^\top,
\label{eq:Iavv}\\
\Theta_{{\rm a},xx}^{k}
&:=
\left[
\Delta\mathbb X_0-2\delta_k\mathbb I_{xx,0},
\ldots,
\Delta\mathbb X_{s-1}-2\delta_k\mathbb I_{xx,s-1}
\right]^\top.
\label{eq:Theta-a-xx}
\end{align}
Here \(\mathbb I_{xx,\ell}\) denotes the \(\ell\)-th row block in
\(\mathbb I_{xx}\). Equation~\eqref{eq:phase-I-integral-identity} has the regression form
\begin{equation}
\Theta_{{\rm a},A}^{k}\theta_{{\rm a},k}
=
b_{{\rm a},k},
\label{eq:phase-I-regression}
\end{equation}
where
\begin{align}
\Theta_{{\rm a},A}^{k}
&:=
\left[
\Theta_{{\rm a},xx}^{k},
\;
-2\mathbb I_{{\rm a},xv}^{k},
\;
-\mathbb I_{{\rm a},vv}^{k}
\right],
\label{eq:Theta-a-A}\\
\theta_{{\rm a},k}
&:=
\operatorname{col}\!\big(
\operatorname{vech}(P_{\rm a}^{k}),
\operatorname{vec}(\Xi_{\rm a}^{k+1}),
\operatorname{vech}(\Gamma_{\rm a}^{k})
\big),
\label{eq:theta-a-k}\\
b_{{\rm a},k}
&:=
-\mathbb I_{xx}\operatorname{vech}(Q_{{\rm a},k}).
\label{eq:b-a-k}
\end{align}

Let $d_x:=\frac{n(n+1)}{2},\
d_u:=\frac{m(m+1)}{2},\
d:=d_x+mn+d_u.$
\[
\Delta X:=
\begin{bmatrix}
\Delta X_0^\top\\
\vdots\\
\Delta X_{s-1}^\top
\end{bmatrix}
\in\mathbb R^{s\times d_x}.
\]

\begin{asm}
\label{ass:phase-I-data-richness}
For every Phase~I iteration for which the expectation terms in the
regression equation are evaluated, the matrix
\[
\mathscr D_{{\rm a},k}
:=
\begin{bmatrix}
\Delta X &
I_{xx} &
I_{{\rm a},xv}^{k} &
I_{{\rm a},vv}^{k}
\end{bmatrix}
\]
has full column rank:
\begin{equation}
\operatorname{rank}(\mathscr D_{{\rm a},k})
=
2d_x+mn+d_u.
\label{eq:phase-I-data-richness}
\end{equation}
\end{asm}
\begin{proposition}
\label{prop:phase-I-identifiability}
Under Assumption~\ref{ass:phase-I-data-richness}, the matrix
\(\Theta_{{\rm a},A}^{k}\) in~\eqref{eq:phase-I-regression}
has full column rank \(d\). Consequently, the Phase~I regression
equation has the unique solution
\begin{equation}
\theta_{{\rm a},k}
=
\left[
(\Theta_{{\rm a},A}^{k})^\top
\Theta_{{\rm a},A}^{k}
\right]^{-1}
(\Theta_{{\rm a},A}^{k})^\top b_{{\rm a},k}.
\label{eq:phase-I-ls-solution}
\end{equation}
\end{proposition}
\begin{proof}
    See Appendix~\ref{full rank}.
\end{proof}
For finite data, the expectation terms in
\eqref{eq:phase-I-regression} are replaced by sample averages. The
resulting estimate is denoted by \(\widehat\theta_{{\rm a},k}\).
The minimum singular value
\(\sigma_{\min}(\widehat\Theta_{{\rm a},A}^{k})\) and the residual norm
\[
\left\|
\widehat\Theta_{{\rm a},A}^{k}\widehat\theta_{{\rm a},k}
-
\widehat b_{{\rm a},k}
\right\|_2
\]
are used as numerical diagnostics.
\subsection{Data-driven selection of the iteration length in Phase~I}
\label{subsec:phase-I-stepsize}

The Phase~I update changes the cumulative factor from \(c_k\) to
\(c_{k+1}=c_k+\alpha_{k+1}\). The following condition uses only the
Phase~I quantities obtained from the regression.

\begin{lemma}[Phase~I model-free step-length guarantee]
\label{lem:phase-I-step}
Suppose that
\begin{equation}
s\!\big(\mathcal L^{\rm a}_{K_{\rm a}^{k+1},c_k}\big)<0.
\label{eq:pre-shift-stability}
\end{equation}
Define
\[
Q_{{\rm a},k+1}
:=
Q_{\rm a}
+
(K_{\rm a}^{k+1})^\top R_{\rm a}K_{\rm a}^{k+1}.
\]
Let \(\overline P_{\rm a}^{k+1}>0\) be the unique solution of
\begin{equation}
\begin{aligned}
0=&
A_{\rm prev}^\top\overline P_{\rm a}^{k+1}
+
\overline P_{\rm a}^{k+1}A_{\rm prev}\\
&+
(C-DK_{\rm a}^{k+1})^\top
\overline P_{\rm a}^{k+1}
(C-DK_{\rm a}^{k+1})
+
Q_{{\rm a},k+1},
\end{aligned}
\label{eq:auxiliary-comparison-lyapunov}
\end{equation}
where $A_{\rm prev}
=
A-\frac{\bar r}{2}I_n+c_kI_n-BK_{\rm a}^{k+1}.$ 
Then
\begin{equation}
0<\overline P_{\rm a}^{k+1}\leq P_{\rm a}^{k}.
\label{eq:auxiliary-comparison-order}
\end{equation}
If
\begin{equation}
0<\alpha_{k+1}
<
\frac{\lambda_{\min}(Q_{{\rm a},k+1})}
{2\lambda_{\max}(P_{\rm a}^{k})},
\label{eq:model-free-alpha}
\end{equation}
then
\[
s\!\big(\mathcal L^{\rm a}_{K_{\rm a}^{k+1},c_{k+1}}\big)<0,
\quad
c_{k+1}=c_k+\alpha_{k+1}.
\]
\end{lemma}

\begin{proof}
See Appendix~\ref{pf-phase-I-step}.
\end{proof}

Lemma~\ref{lem:phase-I-step} applies to Phase~I. The positive definiteness of
\(Q_{{\rm a},k+1}\) follows from \(Q_{\rm a}>0\) and \(R_{\rm a}>0\).
Phase~II uses the original indefinite weights $Q,R$ and follows
Theorem~\ref{thm:indefinite-PI}.

\subsection{Data-driven indefinite PI in Phase~II}
\label{subsec:phase-II-regression}

Phase~II starts from the stabilizing gain \(K^\dagger\) returned by
Phase~I. The auxiliary quantities \(\bar r\), \(c_k\), \(\alpha_k\),
\(Q_{\rm a}\), and \(R_{\rm a}\) are removed. Set
\[
K^0:=K^\dagger,
\quad
i:=0.
\]
For each \(i\ge0\), define
\begin{align}
A^{i}
&:=A-BK^{i},\\
C^{i}
&:=C-DK^{i},
\label{eq:phase-II-closed-loop}\\
v^{i}(t)
&:=u(t)+K^{i}x(t),
\label{eq:phase-II-virtual-input}\\
Q_{i}
&:=Q+(K^{i})^\top RK^{i},
\label{eq:Qo-i}\\
\Gamma^{i}
&:=D^\top P^{i}D,
\label{eq:Gamma-o-i}\\
\widetilde K^{i+1}
&:=B^\top P^{i}+D^\top P^{i}C
=
\big(R+\Gamma^{i}\big)K^{i+1},
\label{eq:Ktilde-o-i}\\
\Xi^{i+1}
&:=\widetilde K^{i+1}-\Gamma^{i}K^{i}.
\label{eq:Xi-o-i}
\end{align}
The policy update is
\begin{equation}
K^{i+1}
=
\big(R+\Gamma^{i}\big)^{-1}
\big(\Xi^{i+1}+\Gamma^{i}K^{i}\big).
\label{eq:Ko-update-data}
\end{equation}

The original policy-evaluation equation is
\begin{equation}
(A^{i})^\top P^{i}
+
P^{i}A^{i}
+
(C^{i})^\top P^{i}C^{i}
+
Q_{i}
=
0 .
\label{eq:phase-II-PE}
\end{equation}
It\^{o}'s formula gives
\begin{align}
&\mathbb E\!\left[
x(t_{\ell+1})^\top P^{i}x(t_{\ell+1})
-
x(t_\ell)^\top P^{i}x(t_\ell)
\right]
\nonumber\\
&=
\mathbb E\!\int_{t_\ell}^{t_{\ell+1}}
\left[
-x^\top Q_{i}x
+2(v^{i})^\top\Xi^{i+1}x
+(v^{i})^\top\Gamma^{i}v^{i}
\right]\mathrm d\tau .
\label{eq:phase-II-integral-identity}
\end{align}
Define \(\mathbb I_{xv}^{i}\) and \(\mathbb I_{vv}^{i}\) by
\eqref{eq:Iaxv}--\eqref{eq:Iavv} with \(v_{\rm a}^k\) replaced by
\(v^{i}\), and set
\[
\Theta_{xx}:=\left[\Delta\mathbb X_0,\ldots,\Delta\mathbb X_{s-1}\right]^\top.
\]
The original-weight regression equation is 
\begin{equation}
\Theta_{A}^{i}\theta_{i}
=
b_{i},
\label{eq:phase-II-regression}
\end{equation}
where
\begin{align}
\Theta_{A}^{i}
&:=
\left[
\Theta_{xx},
\;
-2\mathbb I_{xv}^{i},
\;
-\mathbb I_{vv}^{i}
\right],
\label{eq:Theta-o-A}\\
\theta_{i}
&:=
\operatorname{col}\!\big(
\operatorname{vech}(P^{i}),
\operatorname{vec}(\Xi^{i+1}),
\operatorname{vech}(\Gamma^{i})
\big),
\label{eq:theta-o-i}\\
b_{i}
&:=
-\mathbb I_{xx}\operatorname{vech}(Q_{i}).
\label{eq:b-o-i}
\end{align}

\begin{asm}
\label{ass:phase-II-data-richness}
In each Phase~II iteration, if the expectation terms in the regression equation are evaluated exactly, then the matrix 
\[
\mathscr D_{i}
:=
\begin{bmatrix}
\Delta X &
I_{xv}^{i} &
I_{vv}^{i}
\end{bmatrix}
\]
has full column rank:
\begin{equation}
\operatorname{rank}(\mathscr D_{i})=d.
\label{eq:phase-II-data-richness}
\end{equation}
\end{asm}

\begin{proposition}
\label{prop:phase-II-identifiability}
Under Assumption~\ref{ass:phase-II-data-richness}, the matrix
\(\Theta_{A}^{i}\) in~\eqref{eq:phase-II-regression}
has full column rank \(d\). Consequently, the Phase~II regression
equation has the unique solution
\begin{equation}
\theta_{i}
=
\left[
(\Theta_{A}^{i})^\top
\Theta_{A}^{i}
\right]^{-1}
(\Theta_{A}^{i})^\top b_{i}.
\label{eq:phase-II-ls-solution}
\end{equation}
\end{proposition}
\begin{proof}
    See Appendix~\ref{full rank}.
\end{proof}
For finite data, the corresponding estimate is denoted by
\(\widehat\theta_{i}\). The quantities
\(\sigma_{\min}(\widehat\Theta_{A}^{i})\) and
\[
\left\|
\widehat\Theta_{A}^{i}\widehat\theta_{i}
-
\widehat b_{i}
\right\|_2
\]
are monitored during the numerical implementation.

\subsection{Main results and algorithm}
\label{subsec:model-free-main}

The same exploratory data set supports the Phase~I and Phase~II regressions. Algorithm~\ref{alg:model-free} implements the finite-data procedure.

\begin{thm}[Conditional Phase~I certificate]
\label{thm:model-free-certificate}
Assume that the expectation terms in the Phase~I regression equations are evaluated exactly. Suppose that Assumption~\ref{ass:phase-I-data-richness} holds, that \(P_{\rm a}^{0}(\bar r)>0\), and that $\alpha_{k+1}$ satisfy~\eqref{eq:model-free-alpha}. If the Phase~I iteration reaches an index \(k_\star\) such that
\begin{equation}
\bar r-2c_{k_\star}\le-\varepsilon,
\label{eq:model-free-stop}
\end{equation}
then
\[
K^\dagger:=K_{\rm a}^{k_\star}
\]
satisfies $s(\mathcal L_{K^\dagger})\le-\varepsilon.$ 
Thus \(K^\dagger\in\mathcal K_{\rm ms}\) is an initial stabilizer for the original indefinite PI.
\end{thm}

\begin{proof}
Proposition~\ref{prop:sufficient-shift} gives $s\!\left(\mathcal L^{\rm a}_{K_{\rm a}^{0},c_0}\right)<0.$

Assumption~\ref{ass:phase-I-data-richness} and exact evaluation of the expectation terms make the Phase~I regression equation equivalent to the equations \eqref{eq:phase-I-PE} and \eqref{eq:Ka-update-data}. By Lemma~\ref{lem:phase-I-step}, for each fixed \(c_k\), the auxiliary positive definite policy update preserves mean-square stability. Hence
\[
s\!\big(\mathcal L^{\rm a}_{K_{\rm a}^{k},c_k}\big)<0,
\quad k=0,1,\ldots.
\]
Proposition~\ref{prop:operator-translation} gives
\[
s\big(\mathcal L_{K_{\rm a}^{k_\star}}\big)
=
s\!\Big(\mathcal L^{\rm a}_{K_{\rm a}^{k_\star},c_{k_\star}}\Big)
+\bar r-2c_{k_\star}
<
\bar r-2c_{k_\star}
\le-\varepsilon.
\]
\end{proof}

\begin{cor}
\label{cor:phase-II-equivalence}
Assume that Theorem~\ref{thm:model-free-certificate} holds. Assume that the expectation terms in the Phase~II regression equations are evaluated exactly and that Assumption~\ref{ass:phase-II-data-richness} holds. Under Assumptions~\ref{ass:ms-stabilizable}--\ref{ass:riccati-feasibility}, the Phase~II regression equation generates the same policy sequence as Theorem~\ref{thm:indefinite-PI}. In particular,
\[
K^{i}\in\mathcal K_{\rm ms},
\
P^\ast\leq P^{i+1}\leq P^{i},
\
P^{i}\to P^\ast,
\
K^{i}\to K^\ast.
\]
\end{cor}

\begin{proof}
Theorem~\ref{thm:model-free-certificate} gives
\(K^{0}=K^\dagger\in\mathcal K_{\rm ms}\).
Assumption~\ref{ass:phase-II-data-richness} and evaluation of the expectation terms make the Phase~II regression equation equivalent to the policy-evaluation and policy-improvement equations in Theorem~\ref{thm:indefinite-PI}. The conclusion follows from Theorem~\ref{thm:indefinite-PI}.
\end{proof}

\begin{algorithm}[!t]
\caption{Two-phase model-free PI for indefinite SLQ control}
\label{alg:model-free}
\begin{algorithmic}[1]
\State \textbf{Input:} Exploratory data; \(Q_{\rm a}>0\), \(R_{\rm a}>0\); original weights \(Q=Q^\top\), \(R=R^\top\); an increasing candidate sequence \(\{\bar r_m\}_{m\ge0}\); \(\alpha_0>0\); \(\varepsilon\ge0\); numerical margins \(\tau_P,\tau_\Theta,\tau_{\rm res},\tau_G>0\).
\State \textbf{Output:} A finite-data candidate stabilizer \(\widehat K^\dagger\) and a finite-data Phase-II policy sequence.
\State Set \(m=0\), \(\widehat K_{\rm a}^0=0\), and \(c_0=\alpha_0\).
\Statex \textbf{Phase I-A: spectral-shift parameter $\bar r$ search}
\Repeat
    \State Set \(\bar r\gets\bar r_m\).
    \State Solve the \(k=0\) Phase~I regression and obtain \(\widehat P_{\rm a}^0(\bar r)\), \(\widehat\Xi_{\rm a}^{1}(\bar r)\), and \(\widehat\Gamma_{\rm a}^{0}(\bar r)\).
    \State Set \(m\gets m+1\).
\Until{\(\lambda_{\min}(\widehat P_{\rm a}^0)\ge\tau_P\), \(\sigma_{\min}(\widehat\Theta_{{\rm a},A}^{0})\ge\tau_\Theta\), and \(\|\widehat\Theta_{{\rm a},A}^{0}\widehat\theta_{{\rm a},0}-\widehat b_{{\rm a},0}\|_2\le\tau_{\rm res}\)}
\Statex \textbf{Phase I-B: auxiliary PI}
\State Set \(k=0\).
\Repeat
    \State Solve the Phase-I regression for \(\widehat P_{\rm a}^k\), \(\widehat\Xi_{\rm a}^{k+1}\), and \(\widehat\Gamma_{\rm a}^{k}\).
    \State Set \(\widehat K_{\rm a}^{k+1}\) using \eqref{eq:Ka-update-data}.
    \State Form \(\widehat Q_{{\rm a},k+1}=Q_{\rm a}+(\widehat K_{\rm a}^{k+1})^\top R_{\rm a}\widehat K_{\rm a}^{k+1}\).
    \State Choose \(\alpha_{k+1}\) from \eqref{eq:model-free-alpha}.
    \State Set \(c_{k+1}=c_k+\alpha_{k+1}\) and \(k\gets k+1\).
\Until{\(\bar r-2c_k<-\varepsilon\)}
\State Set \(\widehat K^\dagger=\widehat K_{\rm a}^{k}\).
\Statex \textbf{Phase II: original indefinite PI}
\State Set \(\widehat K^0=\widehat K^\dagger\) and \(i=0\).
\Repeat
    \State Solve the Phase-II regression for \(\widehat P^{i}\), \(\widehat\Xi^{i+1}\), and \(\widehat\Gamma^{i}\).
    \State Verify \(\lambda_{\min}(R+\widehat\Gamma^{i})\ge\tau_G\).
    \State Set \(\widehat K^{i+1}\) using the finite-data counterpart of \eqref{eq:Ko-update-data}.
    \State Set \(i\gets i+1\).
\Until{\(\|\widehat K^{i}-\widehat K^{i-1}\|_F\) and \(\|\widehat P^{i}-\widehat P^{i-1}\|_F\) are below prescribed tolerances}
\State \Return \(\widehat K^{i}\) and \(\widehat P^{i}\).
\end{algorithmic}
\end{algorithm}

\section{Simulation }\label{VI}
    In this section, A three-state stochastic load-frequency-control model is used to evaluate the model-based algorithm and its finite-data model-free counterpart.

\subsection{System setting}
The numerical study is built on a single-area load-frequency-control (LFC)
model. The model describes the governor--turbine--generator dynamics around an
operating equilibrium and is widely used for frequency regulation studies
\cite{vamvoudakis2015asymptotically}. Let
\[
x(t)=
\begin{bmatrix}
\Delta\alpha(t)\\
\Delta P_{\rm m}(t)\\
\Delta f_{\rm G}(t)
\end{bmatrix},
\qquad
u(t)=\Delta P_{\rm c}(t),
\]
where \(\Delta\alpha\) is the incremental governor-valve position,
\(\Delta P_{\rm m}\) is the incremental mechanical-power output of the
turbine, \(\Delta f_{\rm G}\) is the incremental generator-frequency
deviation, and \(\Delta P_{\rm c}\) is the governor speed-reference command.
The deterministic linearized LFC dynamics are
\begin{equation}
\dot x(t)=A_{\rm c}x(t)+B_{\rm c}u(t),
\label{eq:lfc-deterministic}
\end{equation}
with
\begin{equation}
A_{\rm c}=
\begin{bmatrix}
-\dfrac{1}{T_{\rm g}} & 0 & \dfrac{1}{R_{\rm g}T_{\rm g}}\\[1.2ex]
\dfrac{K_{\rm t}}{T_{\rm t}} & -\dfrac{1}{T_{\rm t}} & 0\\[1.2ex]
0 & \dfrac{K_{\rm p}}{T_{\rm p}} & -\dfrac{1}{T_{\rm p}}
\end{bmatrix},
\qquad
B_{\rm c}=
\begin{bmatrix}
\dfrac{1}{T_{\rm g}}\\[1.2ex]
0\\[0.4ex]
0
\end{bmatrix}.
\label{eq:lfc-matrices}
\end{equation}
Here, \(T_{\rm g}\), \(T_{\rm t}\), and \(T_{\rm p}\) denote the governor,
turbine, and generator--load time constants, respectively. The parameter
\(R_{\rm g}\) is the speed-regulation coefficient. The constants \(K_{\rm t}\)
and \(K_{\rm p}\) are the turbine and generator--load gains.

The stochastic model \eqref{1-2} used in this paper augments
\eqref{eq:lfc-deterministic} with multiplicative disturbances, 
The drift matrices are taken as
\[
A=A_{\rm c},
\qquad
B=B_{\rm c}.
\]
The matrix \(C\) represents proportional random fluctuations in the governor,
turbine, and frequency channels. The vector \(D\) represents command-dependent
uncertainty transmitted to the frequency channel. This construction retains
the physical governor--turbine--generator structure while introducing the
state- and control-dependent multiplicative noise required by the proposed
stochastic control framework.

Following the standard LFC parameterization in
\cite{vamvoudakis2015asymptotically}, the deterministic parameters are selected as
\[
T_{\rm g}=0.08~\mathrm{s},
\qquad
T_{\rm t}=0.10~\mathrm{s},
\qquad
T_{\rm p}=20~\mathrm{s},
\]
\[
R_{\rm g}=2.5,
\qquad
K_{\rm p}=120,
\qquad
K_{\rm t}=1.
\]
Substitution into \eqref{eq:lfc-matrices} gives
\[
A=
\begin{bmatrix}
-12.5 & 0 & 5\\
10 & -10 & 0\\
0 & 6 & -0.05
\end{bmatrix},
\qquad
B=
\begin{bmatrix}
12.5\\
0\\
0
\end{bmatrix}.
\]
The multiplicative-noise matrices are selected as
\[
C=
\begin{bmatrix}
0.10 & 0 & 0\\
0 & 0.06 & 0\\
0 & 0 & 0.02
\end{bmatrix},
\qquad
D=
\begin{bmatrix}
0\\
0\\
0.20
\end{bmatrix}.
\]
The resulting open-loop stochastic Lyapunov operator satisfies
\[
s(\mathcal L_0)=3.4881>0,
\]
so the uncontrolled system \eqref{1-2} is mean-square unstable.

The original state weighting matrix is
\[
Q=\operatorname{diag}(-0.10,\,0.50,\,30.0),\quad R=-0.02.
\]
The large positive weight on \(\Delta f_{\rm G}\) emphasizes frequency
regulation. The negative entry associated with the governor-valve state
creates an indefinite state objective. The auxiliary weights used in Phase~I
are
\[
Q_{\rm a}=0.01I_3,
\quad
R_{\rm a}=0.20.
\]

For the finite-data implementation, one fixed offline data set is collected
and reused in both phases. The data set contains \(N_s=800\) reset experiments.
For each reset experiment, \(10\,000\) Monte-Carlo paths are simulated over
the interval \(h=0.04\) by using \(96\) Euler--Maruyama substeps. The initial
state and the constant exploratory input are sampled independently according
to
\[
x_0\sim\mathcal U([-3,3]^3),
\quad
u_0\sim\mathcal U([-12,12]).
\]
Antithetic Wiener increments are used in the Monte-Carlo sampling. The
Phase-I data-richness matrix has \(16\) columns, and the Phase-II regression
matrix has \(10\) columns. Their numerical ranks, minimum singular values,
and least-squares residuals are monitored throughout the computation. 

\subsection{Model-based spectral-assignment PI}

Algorithm~\ref{alg:model-based} first constructs an initial stabilizer by
using the auxiliary positive definite weights $Q_a, R_a$. The spectral-shift parameter is
selected as
\[
\bar r=s(\mathcal L_0)+10.
\]
Phase~I terminates after \(16\) iterations with
\[
c_{16}=7.2100,
\quad
\bar r-2c_{16}=-0.9319.
\]
The resulting initial stabilizer is
\[
K^\dagger=
\begin{bmatrix}
0.3357 & 0.4553 & 0.8770
\end{bmatrix},
\]
and its closed-loop operator satisfies
\[
s(\mathcal L_{K^\dagger})= -4.2454<0.
\]
Figure~\ref{fig:mb_stage1} displays the designed spectral bound and the
closed-loop spectral abscissa. The auxiliary iteration moves the bound through
the origin and produces a stabilizing gain for the original stochastic system \eqref{sys-k}.

Phase~II applies the indefinite PI. The iteration reaches the stabilizing SARE solution
after \(9\) iterations:
\[
K^\ast=
\begin{bmatrix}
2.2215 & 6.6335 & 19.1482
\end{bmatrix}.
\]
Define Riccati residual $\|\mathcal R(P^j)\|_F = \big\| A^\top P^j+P^jA+C^\top P^jC+Q \allowbreak - L(P^j)^\top G(P^j)^{-1} L(P^j) \big\|_F .$ The final closed-loop spectral abscissa is
\[
s(\mathcal L_{K^\ast})=-21.9130,
\]
and the model-based Riccati residual is
\[
\|\mathcal R(P^\ast)\|_F=2.6605\times10^{-15}.
\]
The admissibility condition remains strict along the iteration, with
\[
\lambda_{\min}(G(P^\ast))=8.4623\times10^{-2}>0.
\]
Figures~\ref{fig:mb_stage2_residual} and
\ref{fig:mb_abs_error} show the decay of the Riccati residual and the
successive policy and value-matrix changes. Figure~\ref{fig:mb_state_traj}
compares the state trajectories with and without feedback. The controller is
activated at \(t=2\,\mathrm{s}\). The controlled trajectories approach the
origin, while the uncontrolled trajectories diverge.

\begin{figure}[!t]
  \centering
  \begin{minipage}[t]{0.48\columnwidth}
    \centering
    \includegraphics[width=\linewidth]{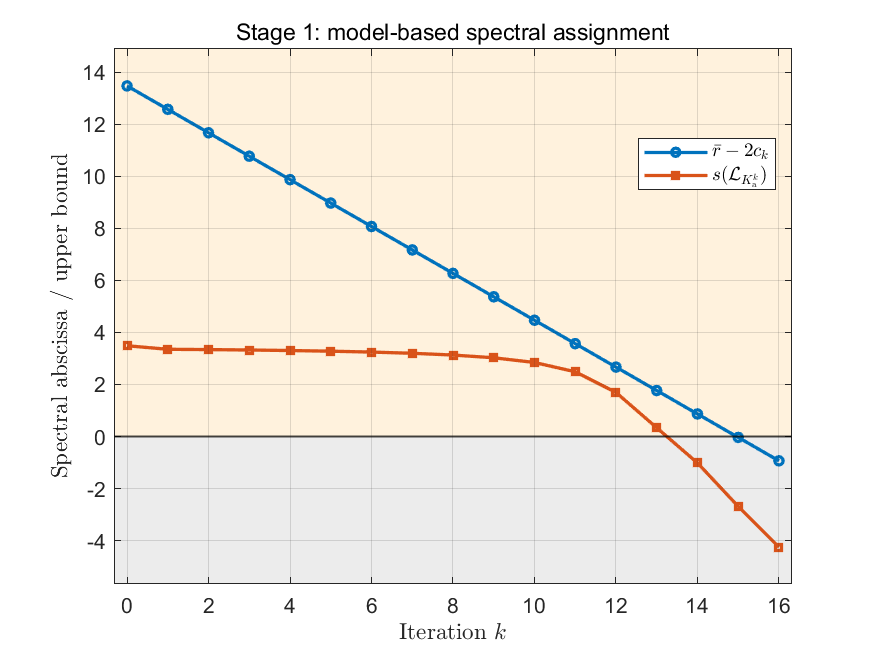}
    \captionof{figure}{Closed-loop spectral abscissa and spectral upper bound in Phase~I of Algorithm~\ref{alg:model-based}.}
    \label{fig:mb_stage1}
  \end{minipage}\hfill
  \begin{minipage}[t]{0.48\columnwidth}
    \centering
    \includegraphics[width=\linewidth]{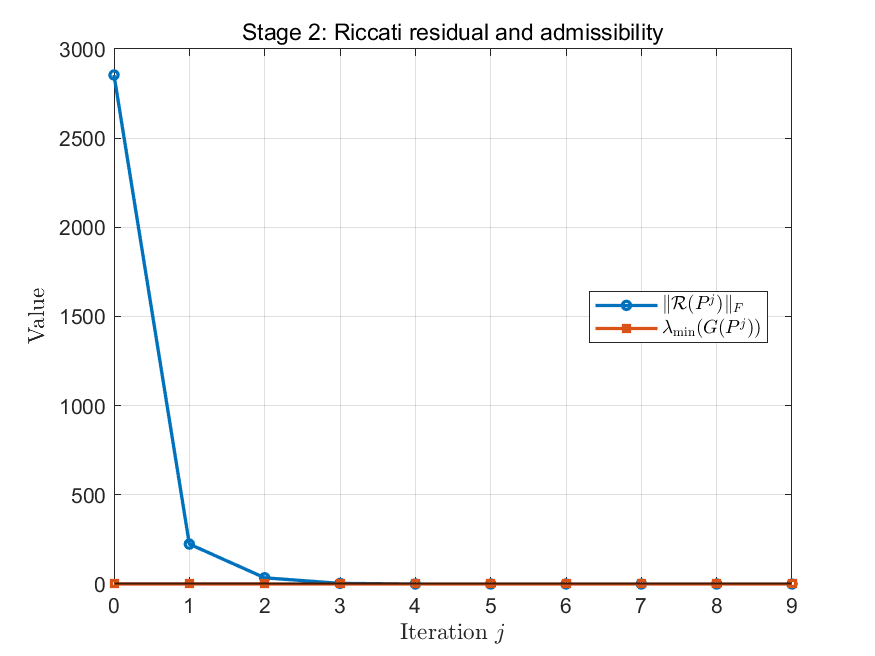}
    \captionof{figure}{Riccati residual and admissibility margin in Phase~II of Algorithm~\ref{alg:model-based}.}
    \label{fig:mb_stage2_residual}
  \end{minipage}
\end{figure}

\begin{figure}[!t]
  \centering
  \begin{minipage}[t]{0.48\columnwidth}
    \centering
    \includegraphics[width=\linewidth]{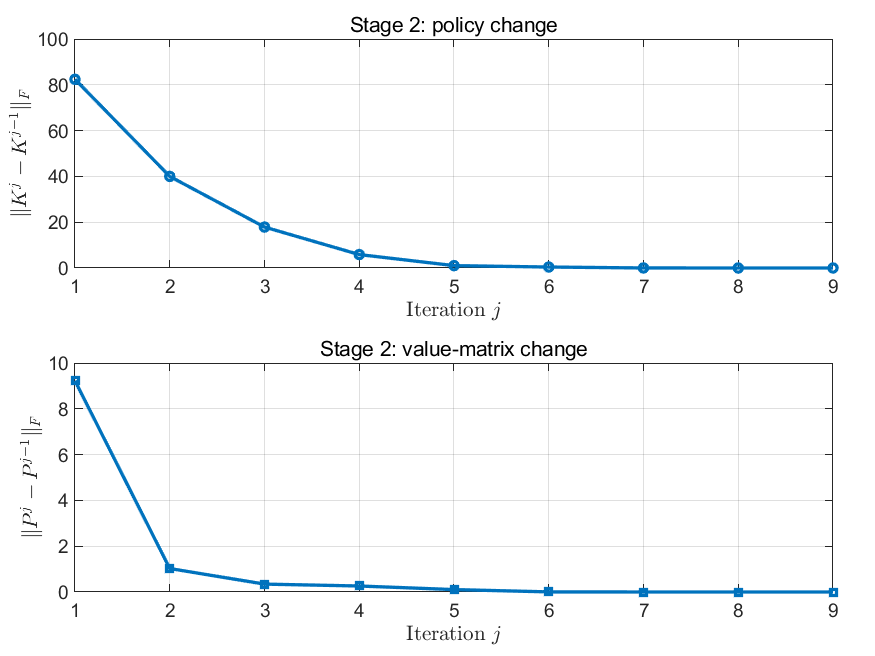}
    \captionof{figure}{Successive policy and value-matrix changes in Phase~II of Algorithm~\ref{alg:model-based}.}
    \label{fig:mb_abs_error}
  \end{minipage}\hfill
  \begin{minipage}[t]{0.48\columnwidth}
    \centering
    \includegraphics[width=\linewidth]{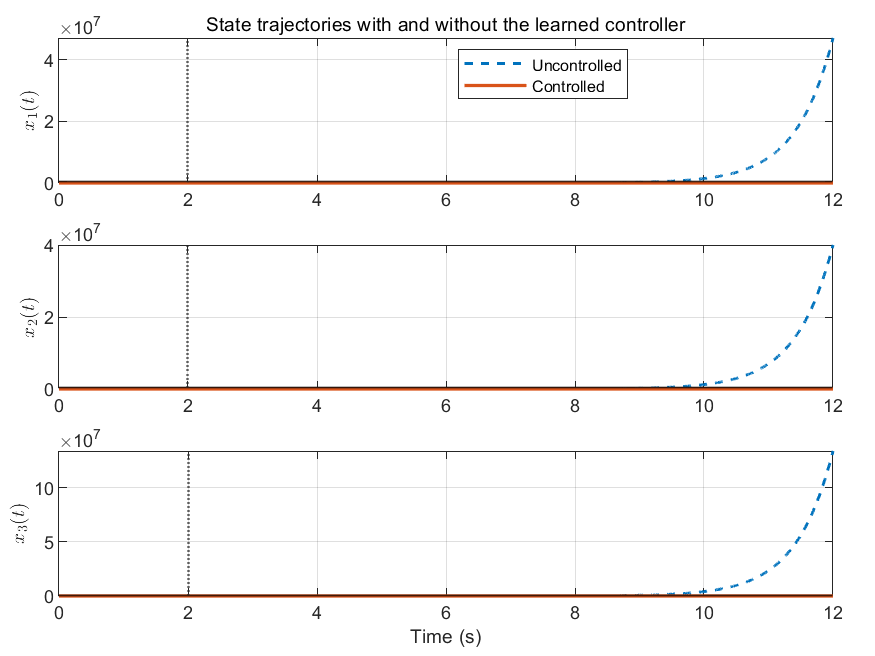}
    \captionof{figure}{State trajectories with and without the controller obtained by Algorithm~\ref{alg:model-based}.}
    \label{fig:mb_state_traj}
  \end{minipage}
\end{figure}

\subsection{Model-free Finite-data PI}

Algorithm~\ref{alg:model-free} uses the fixed offline data set described
above. The candidate search in Phase~I-A selects
\[
\bar r=3.75.
\]
The initial regression produces
\[
\lambda_{\min}(\widehat P_{\rm a}^{0})=2.8195\times10^{-4},
\
\sigma_{\min}(\widehat\Theta_{{\rm a},A}^{0})=2.4519\times10^{-1},
\]
with least-squares residual
\[
\left\|
\widehat\Theta_{{\rm a},A}^{0}\widehat\theta_{{\rm a},0}
-\widehat b_{{\rm a},0}
\right\|_2
=2.0974\times10^{-3}.
\]
Phase~I update reaches its stopping condition after \(25\)
iterations. The terminal quantities are
\[
c_{25}=1.9059,
\quad
\bar r-2c_{25}=-0.0619,
\]
and
\[
\widehat K^\dagger=
\begin{bmatrix}
0.2559 & 0.3420 & 0.6570
\end{bmatrix}.
\]
For validation, the closed-loop operator calculated from the simulation
model satisfies
\[
s(\mathcal L_{\widehat K^\dagger})=-2.3418<0.
\]
Figure~\ref{fig:mf_stage1} shows a smooth decrease of the upper bound and
the closed-loop spectral abscissa.

The
finite-data iteration terminates after \(9\) updates and returns
\[
\widehat K_{\rm mf}=
\begin{bmatrix}
2.2222 & 6.6023 & 18.9504
\end{bmatrix}.
\]
The corresponding validation quantities are
\[
s(\mathcal L_{\widehat K_{\rm mf}})=-22.1241,
\
\lambda_{\min}(R+\widehat\Gamma^{i})=8.6425\times10^{-2},
\]
and the Riccati residual is
\[
\|\mathcal R(\widehat P^{i})\|_F
=3.7767\times10^{-1}.
\]
The nonzero residual measures the discrepancy between the finite-data
regression fixed point and the exact SARE solution. It reflects the
Monte-Carlo approximation, the Euler--Maruyama discretization, and the
least-squares estimation error. Figures~\ref{fig:mf_stage2_residual} and
\ref{fig:mf_abs_error} show that the policy and value-matrix increments
decay to small values. Figure~\ref{fig:mf_state_traj} confirms mean-square
stabilizing behavior in the reported simulation: the controller is activated
at \(t=2\,\mathrm{s}\), the controlled state trajectories remain close to
the origin, and the uncontrolled trajectories diverge.

\begin{figure}[!t]
  \centering
  \begin{minipage}[t]{0.48\columnwidth}
    \centering
    \includegraphics[width=\linewidth]{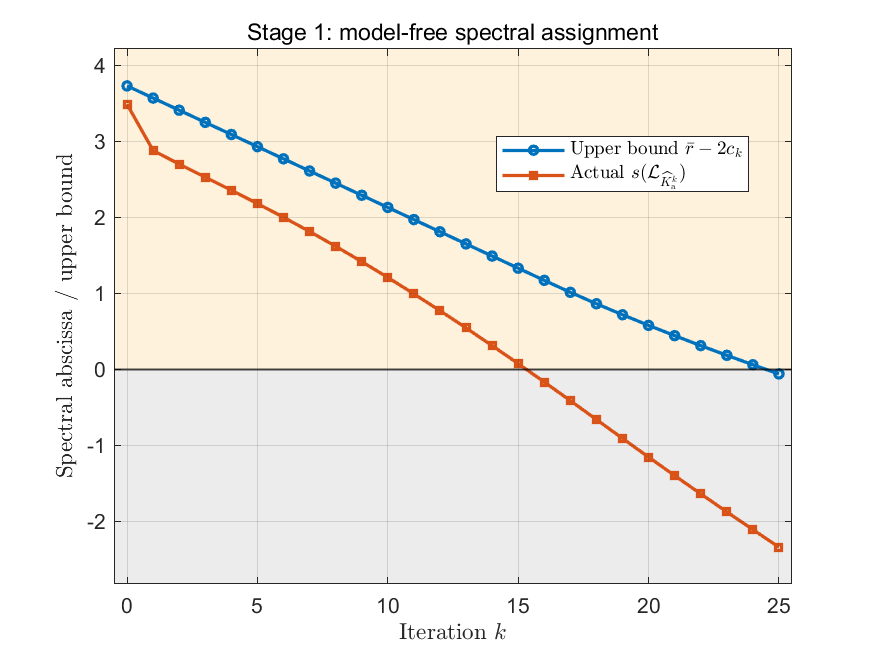}
    \captionof{figure}{Closed-loop spectral abscissa and spectral upper bound in Phase~I of Algorithm~\ref{alg:model-free}.}
    \label{fig:mf_stage1}
  \end{minipage}\hfill
  \begin{minipage}[t]{0.48\columnwidth}
    \centering
    \includegraphics[width=\linewidth]{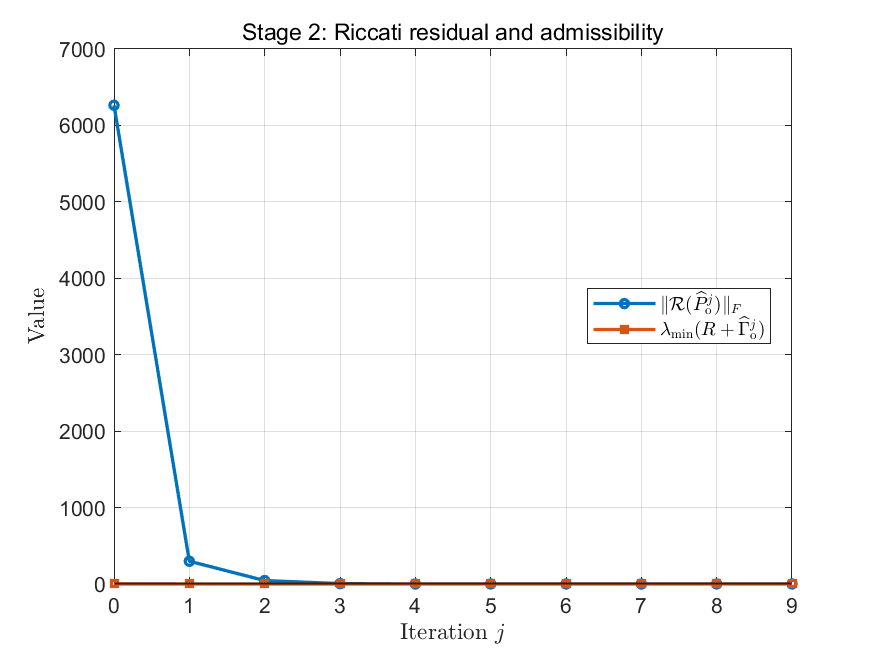}
    \captionof{figure}{Validation Riccati residual and finite-data admissibility margin in Phase~II of Algorithm~\ref{alg:model-free}.}
    \label{fig:mf_stage2_residual}
  \end{minipage}
\end{figure}

\begin{figure}[!t]
  \centering
  \begin{minipage}[t]{0.48\columnwidth}
    \centering
    \includegraphics[width=\linewidth]{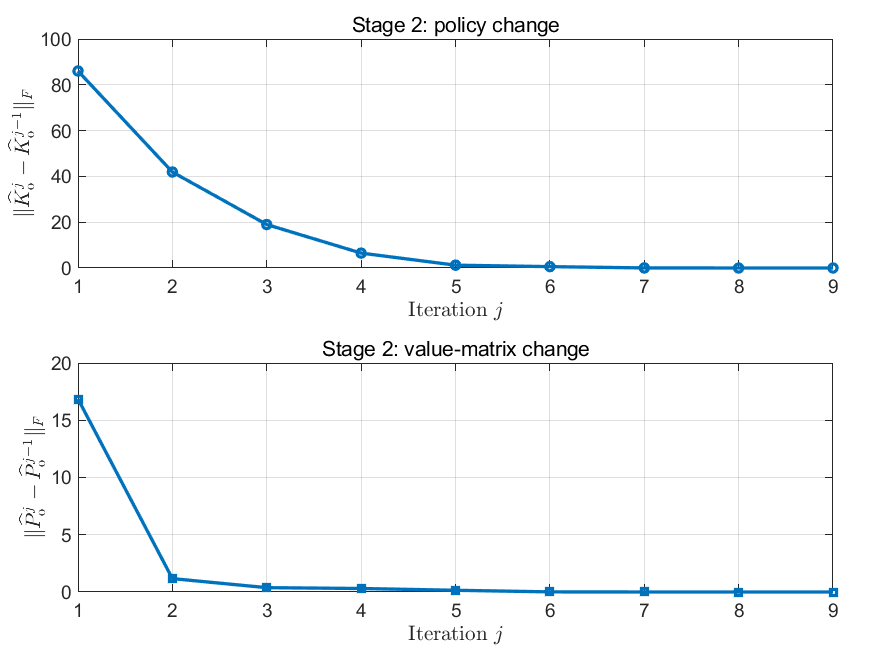}
    \captionof{figure}{Successive finite-data policy and value-matrix changes in Phase~II of Algorithm~\ref{alg:model-free}.}
    \label{fig:mf_abs_error}
  \end{minipage}\hfill
  \begin{minipage}[t]{0.48\columnwidth}
    \centering
    \includegraphics[width=\linewidth]{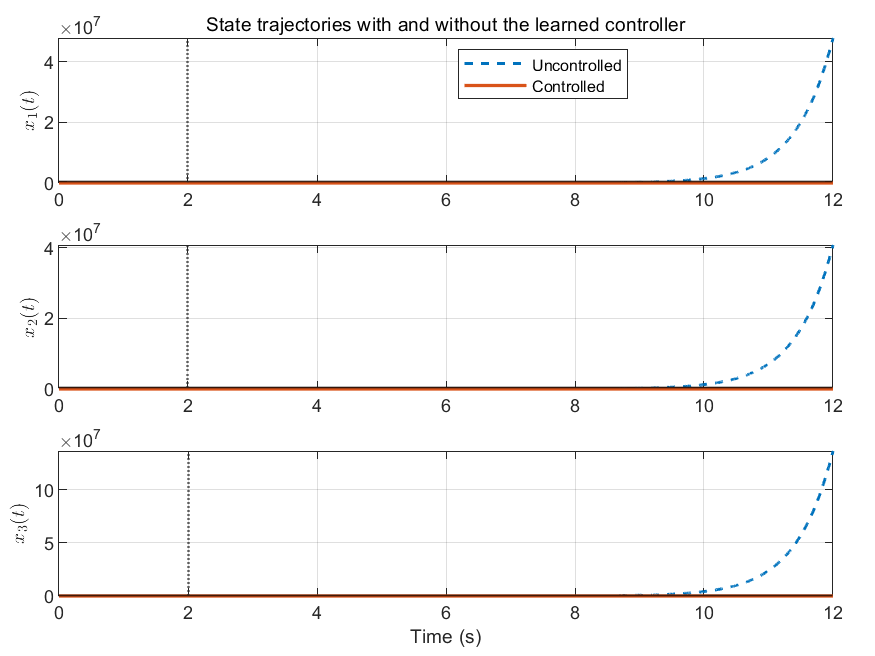}
    \captionof{figure}{State trajectories with and without the controller obtained by Algorithm~\ref{alg:model-free}.}
    \label{fig:mf_state_traj}
  \end{minipage}
\end{figure}
\section{Conclusion }\label{VII}
This paper has developed the PI framework via spectrum assignment for indefinite SLQ control of systems with multiplicative noise. By constructing an auxiliary stable system and updating a cumulative factor, the proposed method moves the operator spectrum into the stable region, and thereby obtains an initial stabilizing controller. This controller is then used to initialize the PI procedure for the original indefinite SLQ problem. Two implementations have been presented. The model-based algorithm gives a constructive stabilizer-search, while the model-free algorithm can approximate the optimal control using fixed offline data without requiring accurate system dynamics. The numerical results show that both algorithms can stabilize an initially unstable stochastic system and converge to the approximately optimal solution. Future work will consider extensions to stochastic \(H_{\infty}\) control, zero-sum games, and nonlinear stochastic systems.

\appendix

\subsection {Proof of Lemma~\ref{lem:policy-identity}}\label{app:proofs}
Substitute $Q_K$, $G(P)$ and $L(P)$  into \eqref{P1.2} and rearrange to obtain
\begin{align}\label{phi-pk}
\begin{split}
    \Phi(P,K)
  =& A^\top P + P A + C^\top P C + Q\\
    &- K^\top L(P) - L(P)^\top K + K^\top G(P) K.
\end{split}
\end{align}
Similarly, substituting $K(P)$ for $K$ into \eqref{P1.2} and expanding yields 
\begin{align}\label{eq:Phi-KP}
\begin{split}
    \Phi(P,K(P))
  = &A^\top P + P A + C^\top P C + Q\\&
    - K(P)^\top L(P) - L(P)^\top K(P)\\&
    + K(P)^\top G(P) K(P).
\end{split}
\end{align}
Subtracting \eqref{eq:Phi-KP} from \eqref{phi-pk} yields
\begin{align*}
\begin{split}
  &\Phi(P,K) - \Phi(P,K(P)) \\
  &=-K^\top L(P) - L(P)^\top K + K^\top G(P) K 
       \\ &\quad \! \ + K(P)^\top L(P) + L(P)^\top K(P)
             - K(P)^\top G(P) K(P), 
\end{split}\label{eq:Phi-diff-1}
\end{align*}
from $K(P) := G(P)^{-1} L(P)$, we have
\begin{equation}\label{eq:Phi-diff-2}
\begin{split}
&\Phi(P,K) - \Phi\bigl(P,K(P)\bigr)\\
&= K^\top G(P)K + K(P)^\top G(P)K(P) \\
&\quad - K^\top G(P)K(P) - K(P)^\top G(P)K \\
&= \bigl(K - K(P)\bigr)^\top G(P)\bigl(K - K(P)\bigr).
\end{split}
\end{equation}
This completes the proof. \hfill $\square$

\subsection{Proof of Lemma~\ref{lem:positive-lyapunov-inverse}}\label{lem:ct_order1} 
From \cite{zhang2004stabilizability}, if matrix $K$ is a stabilizer, the Lyapunov-type operator $\mathcal L_K$ is exponentially stable, i.e.,
\[
\mathrm{Re}\big(\sigma(\mathcal L_K)\big)=\mathrm{Re}\big(\sigma(\mathcal L_K^\ast)\big)<0,
\]
where $0\notin\sigma(\mathcal L_K^\ast)$. Thus, from Theorem~3.1 in \cite{zhang2004stabilizability}, 
for any symmetric matrix $Y=Y^\top$, the equation \eqref{eq:positive-lyapunov} 
admits a unique symmetric solution $X=X^\top$, given by
\[
X=(-\mathcal L_K^\ast)^{-1}Y=\int_0^\infty e^{\,t\mathcal L_K^\ast}(Y)\,dt.
\]
Let $\Phi_K(t)$ denote the state transition matrix of the closed-loop system, for any $t\ge0$,
\[
e^{\,t\mathcal L_K^\ast}(Y)=\mathbb E\!\left[\Phi_K(t)^\top Y\Phi_K(t)\right].
\]
Hence, if $Y\ge 0$, then $e^{\,t\mathcal L_K^\ast}(Y)\ge 0,\quad \forall\,t\ge0$, and therefore $X=\int_0^\infty e^{\,t\mathcal L_K^\ast}(Y)\,dt\ge0$. This completes the proof. \hfill $\square$

\subsection {Proof of Theorem \ref{thm:indefinite-PI}}\label{thm1+prop2}

For brevity, let
\[
\Phi(P,K):=\mathcal L_K^*(P)+Q_K.
\]
By Lemma~\ref{lem:stabilizing-sare}, $G(P^*)>0,
\
K^*=G(P^*)^{-1}L(P^*)\in\mathscr K_{\rm ms},
\
\Phi(P^*,K^*)=0.$

\medskip
\noindent
\emph{Step 1: joint induction for items 1) and 2).}

We prove jointly that, for every \(k\ge0\),
\begin{equation}
K^k\in\mathscr K_{\rm ms},
\quad
P^k-P^*\ge0,
\quad
G(P^k)\ge G(P^*)\>0.
\label{eq:joint-induction-claim}
\end{equation}

\noindent
\emph{Base case: \(P^0\).}
The hypothesis \(K^0\in\mathscr K_{\rm ms}\) ensures that
\eqref{eq:indefinite-PE} has a unique symmetric solution \(P^0\).
Applying Lemma~\ref{lem:policy-identity} at \(P=P^*\) and \(K=K^0\)
gives
\begin{equation}
\Phi(P^*,K^0)
=
(K^0-K^*)^\top G(P^*)(K^0-K^*)
=:Y_0\ge0.
\label{eq:base-Y0}
\end{equation}
Since \(\Phi(P^0,K^0)=0\), subtracting
\eqref{eq:base-Y0} gives
\begin{equation}
\mathcal L_{K^0}^*(P^0-P^*)+Y_0=0.
\label{eq:base-domain-equation}
\end{equation}
Lemma~\ref{lem:positive-lyapunov-inverse} yields $P^0-P^*\succeq0.$

Therefore,
\begin{equation}
G(P^0)
=
G(P^*)+D^\top(P^0-P^*)D
\ge G(P^*)>0.
\label{eq:base-G-positive}
\end{equation}
Thus, \(K^1=G(P^0)^{-1}L(P^0)\) is well defined, and
\eqref{eq:joint-induction-claim} holds at \(k=0\).

\medskip
\noindent
\emph{Induction step: from \((K^k,P^k)\) to
\((K^{k+1},P^{k+1})\).}
Suppose that \eqref{eq:joint-induction-claim} holds for some \(k\ge0\).
In particular, \(G(P^k)>0\); hence, \(K^{k+1}\) in
\eqref{eq:indefinite-PI} is well defined. Define
\begin{align}
Z_k
&:=
(K^k-K^{k+1})^\top G(P^k)(K^k-K^{k+1})
\ge0,
\label{eq:PI-Zk}\\
Y_{k+1}
&:=
(K^{k+1}-K^*)^\top G(P^*)(K^{k+1}-K^*)
\ge0.
\label{eq:PI-Ykplus1}
\end{align}
Since \(K^{k+1}=G(P^k)^{-1}L(P^k)\),
Lemma~\ref{lem:policy-identity} and
\(\Phi(P^k,K^k)=0\) give
\begin{equation}
\Phi(P^k,K^{k+1})=-Z_k.
\label{eq:PI-improvement-identity}
\end{equation}
The same identity at \(P=P^*\), together with
\(\Phi(P^*,K^*)=0\), gives
\begin{equation}
\Phi(P^*,K^{k+1})=Y_{k+1}.
\label{eq:PI-optimal-comparison}
\end{equation}
Subtracting \eqref{eq:PI-optimal-comparison} from
\eqref{eq:PI-improvement-identity} yields
\begin{equation}
\mathcal L_{K^{k+1}}^*(P^k-P^*)+Z_k+Y_{k+1}=0.
\label{eq:PI-stability-identity}
\end{equation}

We first show that \(K^{k+1}\in\mathscr K_{\rm ms}\). Suppose to the
contrary that \(K^{k+1}\notin\mathscr K_{\rm ms}\). By
Lemma~\ref{lem:instability-certificate}, there exist \(\mu\ge0\) and a
nonzero \(X=X^\top\ge0\) such that
\begin{equation}
\mathcal L_{K^{k+1}}(X)=\mu X.
\label{eq:PI-instability-certificate}
\end{equation}
Take the Frobenius inner product of \eqref{eq:PI-stability-identity} with
\(X\). Using the adjoint relation and \eqref{eq:PI-instability-certificate}
gives
\begin{equation}
0
=
\mu\langle X,P^k-P^*\rangle
+\langle X,Z_k\rangle
+\langle X,Y_{k+1}\rangle.
\label{eq:PI-contradiction-sum}
\end{equation}
Every term on the right-hand side of \eqref{eq:PI-contradiction-sum} is
nonnegative. Hence,
\begin{equation}
\langle X,Z_k\rangle=0.
\label{eq:PI-XZ-zero}
\end{equation}
By \eqref{eq:PI-Zk} and \(G(P^k)>0\),
\begin{align*}
\langle X,Z_k\rangle
&=
\operatorname{tr}\!\left(
X(K^k-K^{k+1})^\top G(P^k)(K^k-K^{k+1})
\right)\\
&=
\left\|
G(P^k)^{1/2}(K^k-K^{k+1})X^{1/2}
\right\|_F^2.
\end{align*}
Thus, \eqref{eq:PI-XZ-zero} implies
\begin{equation}
(K^k-K^{k+1})X=0,
\qquad
X(K^k-K^{k+1})^\top=0.
\label{eq:PI-gain-difference-null}
\end{equation}
Using \eqref{eq:PI-gain-difference-null} in the definitions of
\(\mathcal L_{K^k}\) and \(\mathcal L_{K^{k+1}}\) gives
\[
\mathcal L_{K^k}(X)
=
\mathcal L_{K^{k+1}}(X)
=
\mu X.
\]
This contradicts \(K^k\in\mathscr K_{\rm ms}\). Hence, $K^{k+1}\in\mathscr K_{\rm ms}.$

The equation \eqref{eq:indefinite-PE} at iteration \(k+1\) consequently has a
unique symmetric solution \(P^{k+1}\). Combining
\(\Phi(P^{k+1},K^{k+1})=0\) with
\eqref{eq:PI-optimal-comparison} gives
\begin{equation}
\mathcal L_{K^{k+1}}^*(P^{k+1}-P^*)+Y_{k+1}=0.
\label{eq:PI-next-domain-equation}
\end{equation}
Lemma~\ref{lem:positive-lyapunov-inverse} yields $P^{k+1}-P^*\ge0.$

Therefore,
\begin{equation}
G(P^{k+1})
=
G(P^*)+D^\top(P^{k+1}-P^*)D
\ge G(P^*)>0.
\label{eq:PI-next-G-positive}
\end{equation}
This proves \eqref{eq:joint-induction-claim} at \(k+1\). Hence,
 1) and 2) of Theorem~\ref{thm:indefinite-PI} hold for every \(k\ge0\).

\medskip
\noindent
\emph{Step 2: proof of item 3).}
By \eqref{eq:PI-improvement-identity} and
\(\Phi(P^{k+1},K^{k+1})=0\),
\begin{equation}
\mathcal L_{K^{k+1}}^*(P^k-P^{k+1})+Z_k=0.
\label{eq:PI-monotonicity-equation}
\end{equation}
Since \(K^{k+1}\in\mathscr K_{\rm ms}\) and \(Z_k\ge0\),
Lemma~\ref{lem:positive-lyapunov-inverse} yields
\[
P^k-P^{k+1}\ge0.
\]
Together with item 1), this proves
\[
P^*\leq P^{k+1}\leq P^k,
\quad k\ge0.
\]

\medskip
\noindent
\emph{Step 3: proof of item 4).}
By item 3), \(\{P^k\}\) is monotone nonincreasing and bounded below by
\(P^*\). Since \(\mathbb S^n\) is finite dimensional, there exists
\(P^\infty=P^{\infty\top}\) such that
\[
P^k\rightarrow P^\infty,
\quad
P^*\leq P^\infty\leq P^0.
\]
Item 1) gives
\[
G(P^\infty)\ge G(P^*)>0.
\]
By continuity of \(G(\cdot)^{-1}L(\cdot)\),
\[
K^{k+1}\rightarrow
K^\infty:=G(P^\infty)^{-1}L(P^\infty).
\]
The sequence \(\{K^k\}_{k\ge1}\) has the same limit. Passing to
the limit in \(\Phi(P^k,K^k)=0\) gives
\begin{equation}
\Phi(P^\infty,K^\infty)=0.
\label{eq:PI-limit-evaluation}
\end{equation}
Since \(K^\infty=G(P^\infty)^{-1}L(P^\infty)\),
\eqref{eq:PI-limit-evaluation} shows that \(P^\infty\) satisfies the
SARE \eqref{eq:sare}.
\[
A^\top P^\infty+P^\infty A+C^\top P^\infty C+Q
-L(P^\infty)^\top G(P^\infty)^{-1}L(P^\infty)=0.
\]

It remains to verify that \(K^\infty\) is mean-square stabilizing. From
\eqref{eq:PI-limit-evaluation}, the SARE for \(P^*\), and the policy
identity at \(P=P^*\), one obtains
\begin{equation}
\mathcal L_{K^\infty}^*(P^\infty-P^*)
+
(K^\infty-K^*)^\top G(P^*)(K^\infty-K^*)
=0.
\label{eq:PI-limit-stability-identity}
\end{equation}
Suppose that \(K^\infty\notin\mathscr K_{\rm ms}\). Then there exist
\(\mu\ge0\) and a nonzero \(X=X^\top\ge0\) satisfying
\[
\mathcal L_{K^\infty}(X)=\mu X.
\]
Taking the Frobenius inner product of
\eqref{eq:PI-limit-stability-identity} with \(X\) gives
\[
0
=
\mu\langle X,P^\infty-P^*\rangle
+
\left\|
G(P^*)^{1/2}(K^\infty-K^*)X^{1/2}
\right\|_F^2.
\]
Thus, $(K^\infty-K^*)X=0,
\quad
X(K^\infty-K^*)^\top=0.$

Consequently, $\mathcal L_{K^*}(X)
=
\mathcal L_{K^\infty}(X)
=
\mu X,$ 
which contradicts \(K^*\in\mathscr K_{\rm ms}\). Therefore,
\(K^\infty\in\mathscr K_{\rm ms}\).

Hence, \(P^\infty\) is a stabilizing solution of the SARE. The uniqueness
in Lemma~\ref{lem:stabilizing-sare} gives $P^\infty=P^*,
\quad
K^\infty=K^*.$

This proves item 4). \hfill $\square$

\subsection{ Proof of Lemma \ref{lem:auxiliary-spectrum-assignment} and     Proposition~\ref{alpha_bound+Nmax}}\label{auxiliary-spectrum-assignment}

\emph{\textbf{Proof of Lemma~\ref{lem:auxiliary-spectrum-assignment}.}}
Lemma~\ref{lem:initial-auxiliary-stability} gives $s\bigl(\mathcal L^{\rm a}_{K_{\rm a}^{0},c_0}\bigr)<0.$

Suppose that $s\bigl(\mathcal L^{\rm a}_{K_{\rm a}^{k},c_k}\bigr)<0$ 
holds for an arbitrary \(k\ge0\). The stage-cost matrix in
\eqref{eq:auxiliary-PE} is positive definite. Thus, the positive inverse
property of the stable stochastic Lyapunov operator yields
\[
P_{\rm a}^{k}>0.
\]
In particular,
\[
R_{\rm a}+D^\top P_{\rm a}^{k}D>0,
\]
and the update \eqref{eq:auxiliary-PI} is well defined.

Apply Lemma~\ref{lem:policy-identity} to the auxiliary system with \(c_k\)
fixed. Combining \eqref{eq:auxiliary-PE} and
\eqref{eq:auxiliary-PI} gives
\begin{align}
&\bigl(\mathcal L^{\rm a}_{K_{\rm a}^{k+1},c_k}\bigr)^*
(P_{\rm a}^{k})
+Q_{\rm a}
+(K_{\rm a}^{k+1})^\top R_{\rm a}K_{\rm a}^{k+1}
\nonumber\\
&\quad=
-\left(K_{\rm a}^{k}-K_{\rm a}^{k+1}\right)^\top
\left(R_{\rm a}+D^\top P_{\rm a}^{k}D\right)
\left(K_{\rm a}^{k}-K_{\rm a}^{k+1}\right)
\leq0.
\label{eq:auxiliary-improvement-identity}
\end{align}
It follows that
\begin{equation}
\bigl(\mathcal L^{\rm a}_{K_{\rm a}^{k+1},c_k}\bigr)^*
(P_{\rm a}^{k})
\leq
-Q_{\rm a}
\leq
-\frac{\lambda_{\min}(Q_{\rm a})}
{\lambda_{\max}(P_{\rm a}^{k})}P_{\rm a}^{k}.
\label{eq:auxiliary-Lyapunov-margin}
\end{equation}
Let
\[
\eta_k:=
\frac{\lambda_{\min}(Q_{\rm a})}
{\lambda_{\max}(P_{\rm a}^{k})}.
\]
Since \(P_{\rm a}^{k}\leq
\lambda_{\max}(P_{\rm a}^{k})I_n\) and
\(Q_{\rm a}\ge\lambda_{\min}(Q_{\rm a})I_n\), one has
\[
Q_{\rm a}\ge \eta_k P_{\rm a}^{k}.
\]
Hence, \eqref{eq:auxiliary-Lyapunov-margin} is a strict stochastic
Lyapunov inequality. Applying It\^{o}'s formula to
\(x^\top P_{\rm a}^{k}x\) gives
\[
\mathbb E[x(t)^\top P_{\rm a}^{k}x(t)]
\le
e^{-\eta_k t}x(0)^\top P_{\rm a}^{k}x(0),
\]
which yields
\[
s\bigl(\mathcal L^{\rm a}_{K_{\rm a}^{k+1},c_k}\bigr)
\le-\eta_k = -\frac{\lambda_{\min}(Q_{\rm a})}
{\lambda_{\max}(P_{\rm a}^{k})}<0.
\]
Therefore, the step size in \eqref{eq:model-based-alpha} satisfies
\[
0<\alpha_{k+1}
\le
-\frac{\vartheta}{2}
s\bigl(\mathcal L^{\rm a}_{K_{\rm a}^{k+1},c_k}\bigr)
<
-\frac12
s\bigl(\mathcal L^{\rm a}_{K_{\rm a}^{k+1},c_k}\bigr),
\]
which proves \eqref{eq:model-based-alpha-feasibility}.

For the fixed gain \(K_{\rm a}^{k+1}\), Proposition~\ref{prop:operator-translation}
gives
\begin{align*}
s\bigl(\mathcal L^{\rm a}_{K_{\rm a}^{k+1},c_{k+1}}\bigr)
&=
s\bigl(\mathcal L^{\rm a}_{K_{\rm a}^{k+1},c_k}\bigr)
+2\alpha_{k+1}\\
&\le
-(1-\vartheta)
\frac{\lambda_{\min}(Q_{\rm a})}
{\lambda_{\max}(P_{\rm a}^{k})}
<0.
\end{align*}
Thus, \eqref{eq:auxiliary-stability-all-k} follows by induction.

Equation~\eqref{eq:original-spectrum-bound} follows from
\eqref{eq:abscissa-translation}:
\[
s\bigl(\mathcal L_{K_{\rm a}^{k}}\bigr)
=
s\bigl(\mathcal L^{\rm a}_{K_{\rm a}^{k},c_k}\bigr)
+\bar r-2c_k
<
\bar r-2c_k.
\]
At an index \(k_\star\) satisfying \eqref{eq:phase-I-stop}, this relation
gives
\[
s\bigl(\mathcal L_{K_{\rm a}^{k_\star}}\bigr)
<
\bar r-2c_{k_\star}
<-\varepsilon.
\]
Thus, \(K^\dagger=K_{\rm a}^{k_\star}\in\mathscr K_{\rm ms}\).

\medskip
\noindent
\emph{\textbf{Proof of Proposition~\ref{alpha_bound+Nmax}.}}
By \eqref{eq:model-based-alpha} and
\eqref{eq:uniform-PhaseI-value-bound},
\[
\alpha_{k+1}
=
\frac{\vartheta\lambda_{\min}(Q_{\rm a})}
{2\lambda_{\max}(P_{\rm a}^{k})}
\ge
\frac{\vartheta\lambda_{\min}(Q_{\rm a})}{2\bar p}
=
\underline\alpha>0,
\]
which proves \eqref{eq:alpha-lower-bound}. Hence,
\[
c_k
=
c_0+\sum_{j=1}^{k}\alpha_j
\ge
c_0+k\underline\alpha.
\]
For \(k=N_{\max}\), \eqref{eq:PhaseI-iteration-bound} yields 
\[
c_k>
\frac{\bar r+\varepsilon}{2}.
\]
This inequality is equivalent to
\(\bar r-2c_k<-\varepsilon\), namely,
\eqref{eq:phase-I-stop}. Therefore, Phase~I terminates in finitely many
iterations. \hfill\(\square\)

\subsection{Proofs of Propositions~\ref{prop:phase-I-identifiability} and
\ref{prop:phase-II-identifiability}}\label{full rank}

\textbf{Proof of Proposition~\ref{prop:phase-I-identifiability}.}
Write
\[
\mathcal T_{{\rm a},k}
:=
\begin{bmatrix}
I_{d_x} & 0 & 0\\
-2\delta_k I_{d_x} & 0 & 0\\
0 & -2I_{mn} & 0\\
0 & 0 & -I_{d_u}
\end{bmatrix}.
\]
The dimensions of \(\mathcal T_{{\rm a},k}\) are
\[
(2d_x+mn+d_u)\times d.
\]
By the definitions of \(\Theta_{{\rm a},xx}^{k}\) and
\(\Theta_{{\rm a},A}^{k}\),
\[
\Theta_{{\rm a},A}^{k}
=
\mathscr D_{{\rm a},k}\mathcal T_{{\rm a},k}.
\]
Let \(z=\operatorname{col}(z_P,z_\Xi,z_\Gamma)\in\mathbb R^d\)
satisfy
\[
\Theta_{{\rm a},A}^{k}z=0.
\]
Assumption~\ref{ass:phase-I-data-richness} implies that
\(\mathscr D_{{\rm a},k}\) has trivial null space. Hence
\[
\mathcal T_{{\rm a},k}z=0.
\]
The first block row gives \(z_P=0\). The third and fourth block rows
then give \(z_\Xi=0\) and \(z_\Gamma=0\), respectively. Thus
\(z=0\), which proves
\[
\operatorname{rank}(\Theta_{{\rm a},A}^{k})=d.
\]
The exact regression identity is consistent with the true parameter
vector. Full column rank yields the unique least-squares solution
\eqref{eq:phase-I-ls-solution}. \hfill\(\square\)

\textbf{Proof of Proposition~\ref{prop:phase-II-identifiability}.}
Define
\[
\mathcal T_{\rm o}
:=
\operatorname{diag}
\left(I_{d_x},-2I_{mn},-I_{d_u}\right).
\]
The matrix \(\mathcal T_{\rm o}\in\mathbb R^{d\times d}\) is
nonsingular. From the definition of the Phase-II regression matrix,
\[
\Theta_{A}^{i}
=
\mathscr D_{i}\mathcal T_{\rm o}.
\]
Assumption~\ref{ass:phase-II-data-richness} gives
\[
\operatorname{rank}(\mathscr D_{i})=d.
\]
Right multiplication by the nonsingular matrix \(\mathcal T_{\rm o}\)
preserves column rank. Hence
\[
\operatorname{rank}(\Theta_{A}^{i})=d.
\]
The exact regression identity is consistent with the true parameter
vector, and the unique least-squares solution is
\eqref{eq:phase-II-ls-solution}. \hfill\(\square\)

\subsection{Proof of Lemma~\ref{lem:phase-I-step}}\label{pf-phase-I-step}

From Theorem~2.1 in \cite{li2022stochastic}, the stability condition~\eqref{eq:pre-shift-stability} and
\(Q_{{\rm a},k+1}>0\) imply that
\eqref{eq:auxiliary-comparison-lyapunov} has a unique solution
\(\overline P_{\rm a}^{k+1}>0\).

Define
\[
Z_{{\rm a},k}
:=
(K_{\rm a}^{k}-K_{\rm a}^{k+1})^\top
\left(R_{\rm a}+D^\top P_{\rm a}^{k}D\right)
(K_{\rm a}^{k}-K_{\rm a}^{k+1})
\ge0.
\]
The auxiliary policy-evaluation equation at \(K_{\rm a}^{k}\), the policy-improvement identity, and
\eqref{eq:auxiliary-comparison-lyapunov} yield
\[
\mathcal L_{K_{\rm a}^{k+1},c_k}^{{\rm a},*}
\left(P_{\rm a}^{k}-\overline P_{\rm a}^{k+1}\right)
=
-Z_{{\rm a},k}
\leq0.
\]
The operator
\(\mathcal L_{K_{\rm a}^{k+1},c_k}^{\rm a}\) is stable by
\eqref{eq:pre-shift-stability}. From Lemma~\ref{lem:positive-lyapunov-inverse}, we have
\[
P_{\rm a}^{k}-\overline P_{\rm a}^{k+1}\ge0,
\]
which proves~\eqref{eq:auxiliary-comparison-order}.

The change from \(c_k\) to \(c_{k+1}\) gives
\[
\mathcal L_{K_{\rm a}^{k+1},c_{k+1}}^{{\rm a},*}
\left(\overline P_{\rm a}^{k+1}\right)
=
-Q_{{\rm a},k+1}
+
2\alpha_{k+1}\overline P_{\rm a}^{k+1}.
\]
By~\eqref{eq:auxiliary-comparison-order} and
\eqref{eq:model-free-alpha},
\[
-Q_{{\rm a},k+1}
+
2\alpha_{k+1}\overline P_{\rm a}^{k+1}
<0.
\]
By Lemma~2.1 in \cite{li2022stochastic}, we conclude
\[
s\!\left(\mathcal L^{\rm a}_{K_{\rm a}^{k+1},c_{k+1}}\right)<0.
\] \hfill\(\square\)

\end{document}